\documentclass[reqno,final]{amsart}
\usepackage{natbib}  
\usepackage{fancyhdr} 
\usepackage{color} 
\usepackage{hyperref} 
\usepackage{graphicx} 

\usepackage{hyperref}
\usepackage{url}

\usepackage{xargs}

\usepackage[textwidth=4cm]{todonotes}

\RequirePackage[OT1]{fontenc}
\RequirePackage{amsthm,amsmath,txfonts,bm,pifont,graphicx,bbm,enumitem,xr}
\RequirePackage[utf8]{inputenc}
\RequirePackage[norelsize,ruled,vlined]{algorithm2e}
\SetKwHangingKw{KwParameters}{parameters}
\SetKwHangingKw{KwInit}{initialization}
\SetKwHangingKw{KwIn}{input}
\SetKwFor{KwAndWhen}{and when}{do}{endfor}

\definecolor{aleacolor}{rgb}{0.16,0.59,0.78}

\hypersetup{
breaklinks,
colorlinks=true,
linkcolor=aleacolor,
urlcolor=aleacolor,
citecolor=aleacolor}

\renewcommand{\cite}{\citet}

\theoremstyle{plain}
\newtheorem{theorem}{Theorem}[section]                                          
                          
\newtheorem{lemma}[theorem]{Lemma}
\newtheorem{corollary}[theorem]{Corollary}

\theoremstyle{definition}
\newtheorem{definition}[theorem]{Definition}
\theoremstyle{remark}
\newtheorem{remark}[theorem]{Remark}
\newtheorem{example}{Example}
\newenvironment{hypothesis}[1]{
	\begin{enumerate}[label={\sf(\textbf{#1}-\arabic*)},resume=hyp#1]\begin{sf}}
		{\end{sf}\end{enumerate}}
\newenvironment{hypothesis*}[1]{
	\begin{enumerate}[label={\sf(\textbf{#1})}]\begin{sf}}
		{\end{sf}\end{enumerate}}

\makeatletter \@addtoreset{equation}{section} \makeatother

\newcommand{\rme}{\mathrm{e}}
\newcommand{\N}{\mathbb{N}}
\newcommand{\Z}{\mathbb{Z}}
\newcommand{\zset}{\mathbb{Z}}
\newcommand{\R}{\mathbb{R}}
\newcommand{\C}{\mathbb{C}}
\newcommand{\E}{\mathbb{E}}
\newcommand{\EE}{\mathbb{E}}
\newcommand{\PP}{\mathbb{P}}

\newcommand{\btheta}{\bm\theta}
\newcommand{\Dfrac}[2]{{\displaystyle \frac{#1}{#2}}}
\newcommand{\rmd}{\mathrm{d}}
\newcommand{\rmi}{\mathrm{i}}

\begin{document}

\hypersetup{pageanchor=false}

\title[Prediction of weakly locally stationary processes]
{Prediction of weakly locally stationary processes by auto-regression}

\author{Fran\c{c}ois~Roueff}
\author{Andrés~S\'anchez-Pérez}

\address{LTCI, Télécom ParisTech, Universit\'e Paris-Saclay\newline
46, rue Barrault,\newline
75634 Paris Cedex 13, France.}

\email{francois.roueff@telecom-paristech.fr}
\urladdr{\url{https://perso.telecom-paristech.fr/~roueff/}}

\thanks{This work has been partially supported by the Conseil r\'egional
d'\^{I}le-de-France under a doctoral allowance of its program R\'eseau de
Recherche Doctoral en Math\'ematiques de l'\^{I}le de France (RDM-IdF) for the period 2012 - 2015 and by the Labex LMH (ANR-11-IDEX-003-02).
}


\begin{abstract}
  In this contribution we introduce weakly locally stationary time series
  through the local approximation of the non-stationary covariance structure by
  a stationary one.  This allows us to define autoregression coefficients in a
  non-stationary context, which, in the particular case of a locally stationary
  Time Varying Autoregressive (TVAR) process, coincide with the generating
  coefficients. We provide and study an estimator of the time varying
  autoregression coefficients in a general setting. The proposed estimator of
  these coefficients enjoys an optimal minimax convergence rate under limited
  smoothness conditions. In a second step, using a bias reduction technique, we
  derive a minimax-rate estimator for arbitrarily smooth time-evolving
  coefficients, which outperforms the previous one for large data sets. In
  turn, for TVAR processes, the predictor derived from the estimator exhibits
  an optimal minimax prediction rate.
\end{abstract}

\maketitle

\hypersetup{pageanchor=true}
\section{Introduction}

In many applications, one is interested in predicting the next values of an
observed time series. It is the case in various areas like finance (stock
market, volatility on prices), social sciences (population studies),
epidemiology, meteorology and network systems (Internet
traffic). Autoregressive processes have been used successfully in a stationary
context for several decades.  On the other hand, in a context where the number
of observations can be very large, the usual stationarity assumption has to be
weakened to take into account some smooth evolution of the environment. 

Several prediction methods developed in signal processing are well known to adapt
to a changing environment. This is the case of the wide spread recursive least
square algorithms. The initial goal of these methods is to provide an online
procedure for estimating a regression vector with low numerical cost. Such
methods usually rely on a forgetting factor or a gradient step size $\gamma$
and they can be shown to be consistent in a stationary environment when
$\gamma$ decreases adequately to zero (see e.g. \cite{Duflo:1997}). Even when
the environment is changing, that is, when the regression parameter evolves
along the time, a ``small enough'' $\gamma$ often yields a good tracking of the
evolving regression parameter. In order to have a sound and comprehensive
understanding of this phenomenon, an interesting approach is to consider a
local stationarity assumption, as successfully initiated in
\cite{Dahlhaus:1996} by relying on a non-stationary spectral representation
introduced in \cite{Priestley:1965}; see also \cite{Dahlhaus:2012} and the
references therein for a recent overview. The basic idea is to provide an
asymptotic analysis for the statistical inference of non-stationary time series
such as time varying autoregressive (TVAR) processes using local
stationary approximations. The analysis of the Normalized Least Mean Squares
(NLMS) algorithm for tracking a changing autoregression parameter in this
framework is tackled in \cite{Moulines_Priouret_Roueff:2005}.  Such an analysis
is based on the usual tools of non-parametric statistics.  The TVAR parameter
$\btheta$ is seen as the regular samples of a smooth $\R^{d}$-valued
function. An in-fill asymptotic allows one to derive the rates of
convergence of the NLMS estimator for estimating this function within
particular smoothness classes of functions.  As shown in
\cite{Moulines_Priouret_Roueff:2005}, it turns out that the NLMS algorithm
provides an optimal minimax rate estimator of the TVAR parameter with Hölder
smoothness index $\beta\in(0,1]$. However it is no longer optimal for $\beta>1$, that
is, when the TVAR parameter is smoother than a continuously differentiable
function. An improvement of the NLMS is proposed in
\cite{Moulines_Priouret_Roueff:2005} to cope with the case $\beta\in(0,2]$ but,
to the best of our knowledge, there is no available method neither for the
$\btheta$ minimax-rate estimation nor for the minimax-rate prediction when
$\beta>2$, that is when the TVAR parameter is smoother than a two-times
continuously differentiable function.

In the present work, our main contribution is twofold. First we extend the
concept of time-varying linear prediction coefficients to a general class of
weakly locally stationary processes, which includes the class of locally
stationary processes as introduced in \cite{Dahlhaus:1996}.  In the specific
case of a TVAR process, these coefficients correspond to the time-varying
autoregression parameters. Second, we show that the tapered Yule-Walker
estimator introduced in \cite{Dahlhaus_Giraitis:1998} for TVAR processes also
applies to this general class and is minimax-rate for Hölder indices up to
$\beta=1$ for asymmetric tapers and up to $\beta=2$ for symmetric
ones. Moreover, by applying a bias reduction technique, we derive a new
estimator which is minimax-rate for any arbitrarily large Hölder index $\beta$.
By achieving this goal, we provide a theoretically justified construction of
predictors that can be chosen optimally, depending on how smoothly the time
varying spectral density evolves along the time. On the other hand, in
practical situations, one may not have a clear view on the value of the
smoothness index $\beta$ and one should rely on data driven methods that are
therefore called \emph{adaptive}. This problem was recently tackled in
\cite{Giraud_Roueff_Sanchez-Perez:2015} . More precisely, using aggregation
techniques introduced in the context of individual sequences prediction (see
\emph{e.g.} \cite{Cesa-Bianchi_Lugosi:2006}) and statistical learning (see
\emph{e.g.}  \cite{Barron:1987}), one can aggregate sufficiently many
predictors in order to build a minimax predictor which adapts to the unknown
smoothness $\beta$ of the time varying parameter of a TVAR process.  However, a
crucial requirement in \cite{Giraud_Roueff_Sanchez-Perez:2015} is to rely on
$\beta$-minimax-rate sequences of predictors for any $\beta>0$. Our main
contribution here is to fill this gap, hence achieving to solve the problem of
the adaptive minimax-rate linear forecasting of locally stationary TVAR
processes with coefficients of any (unknown, arbitrarily large) Hölder
smoothness index.

The paper is organized as follows. In
Section~\ref{section:locally_stationary_time_series}, we introduce a definition
of weakly locally stationary time series, the
regression problem investigated in this work in relation with the practical
prediction problem, and the tapered Yule-Walker estimator under study. 
General results on this estimator are presented in
Section~\ref{section:main_results} and a minimax rate estimator is derived. The
particular case of TVAR processes is treated in
Section~\ref{section:TVAR}. Numerical experiments illustrating these results
can be found in Section~\ref{section:numerical_work_minimax}. Postponed proofs
and useful lemmas are provided in the appendices.

\section{General setting} \label{section:locally_stationary_time_series} 
In the
following, non-random vectors and sequences are denoted using boldface symbols,
$\|\mathbf{x}\|$ denotes the Euclidean norm of $\mathbf{x}$,
$\|\mathbf{x}\|=(\sum_i|x_i|^2)^{1/2}$, and $\|\mathbf{x}\|_{1}$ its $\ell_{1}$
norm, $\|\mathbf{x}\|_{1}=\sum_i|x_i|$. If $f$ is a function,
$\|f\|_{\infty}=\sup_{x}|f(x)|$ corresponds to its sup norm.  If $A$ is a
matrix, $\|A\|$ denotes its spectral norm,
$\|A\|=\sup\{\|A\mathbf{x}\|\,,\;\|\mathbf{x}\|\leq1\}$. We moreover denote 
$$
\ell^1(\N)=\{\bm x\in\R^\N\;\text{s.t.}\;\|\bm x\|_1<\infty\} 
\quad\text{and}\quad
\ell_+^1(\N)=\{\bm x\in\R_+^\N\;\text{s.t.}\;\|\bm x\|_1<\infty\} \;.
$$
\subsection{Main definitions}
We consider a doubly indexed time series $(X_{t,T})_{t\in\Z,T\in\N^*}$. Here $t$ refers to a discrete time
index and $T$ is an additional index indicating the sharpness of the
\emph{local approximation} of the time series $(X_{t,T})_{t\in\Z}$ by a
stationary one. Coarsely speaking, $(X_{t,T})_{t\in\Z,T\in\N^*}$ is considered
to be \emph{weakly locally stationary} if, for $T$ large, given a set $S_T$ of sample
indices such that $t/T\approx u$ over $t\in S_T$, the sample $(X_{t,T})_{t\in
  S_T}$ can be approximately viewed as the sample of a  weakly stationary time series
depending on the \emph{rescaled location} $u$. Note that $u$ is a
continuous time parameter, sometimes referred to as the \emph{rescaled time}
index. Following \cite{Dahlhaus:1996}, $T$ is usually interpreted as the
number of available observations, in which case all the definitions are
restricted to $1\leq t\leq T$ and $u\in[0,1]$. However this is not essential in
our mathematical derivations and it is more convenient to set $t\in\Z$ and
$u\in\R$ for presenting our setting.

We use the following class of functions. For
$\alpha\in(0,1]$ the $\alpha-$H\"{o}lder semi-norm of a function
$\mathbf{f}:\R\rightarrow\C^{d}$ is defined by
$$
|\mathbf{f}|_{\alpha,0} = \sup_{0<|s-s'|<1}\Dfrac{\left\|\mathbf{f}(s)-\mathbf{f}(s')\right\|}{|s-s'|^{\alpha}} \;.
$$
This semi-norm is used to build a norm for any $\beta>0$ as it follows. Let
$k\in\N$ and $\alpha\in(0,1]$ be such that $\beta=k+\alpha$. If $\mathbf{f}$ is
$k$ times differentiable on $\R$, we define
$$
|\mathbf{f}|_{\beta} = \left|\mathbf{f}^{(k)}\right|_{\alpha,0}+\max_{0\leq s\leq k}\left\|\mathbf{f}^{(s)}\right\|_{\infty} \;,
$$
and  $|\mathbf{f}|_{\beta}=\infty$ otherwise. For $R>0$ and
$\beta>0$, the $(\beta,R)-$
H\"{o}lder ball of dimension $d$ is denoted by
$$
\Lambda_{d}(\beta,R) = \left\{\mathbf{f}:\R\rightarrow\C^{d},\ \textrm{such that}\ \ |\mathbf{f}|_{\beta}\leq R\right\} \;.
$$
We first introduce definitions for the time varying covariance and the local covariance functions. 
\begin{definition}[Time varying covariance function]
	Let $(X_{t,T})_{t\in\Z,T\in\N^{*}}$ be an array of random variables with
	finite variances. The local time varying covariance function $\gamma^{*}$ is
	defined for all $t\in\Z, T\in\N^{*}$ and $\ell\in\Z$ as
	\begin{eqnarray}
		\gamma^{*}\left(t,T,\ell\right) = \mathrm{cov}\left(X_{t,T},X_{t-\ell,T}\right) \;. \label{equation:time_varying_covariance_function}
	\end{eqnarray}
\end{definition}
\begin{definition}[Local covariance function and local spectral density]
  A \emph{local spectral density} 
  $f$ is a $\R^2\to\R_{+}$ function, $(2\pi)$-periodic and locally integrable with
  respect to the second variable. The \emph{local covariance function} $\gamma$
        associated with  the  \emph{local spectral density} $f$ is defined on $\R\times\Z$ by
	\begin{eqnarray}
		\gamma\left(u,\ell\right) = \int_{-\pi}^{\pi}\rme^{\rmi\ell\lambda}f\left(u,\lambda\right)\rmd\lambda \;. \label{equation:local_covariance_function}
	\end{eqnarray}
\end{definition}
In~(\ref{equation:local_covariance_function}), the variable $u$ should be seen
as \emph{rescaled} time index (in $\R$), $\ell$ as a (non-rescaled) time index
and $\lambda$ as a frequency (in $[-\pi,\pi]$). 
Recall that, by the Herglotz theorem (see
\cite[Theorem~4.3.1]{Brockwell_Davis:2002}), Equation~(\ref{equation:local_covariance_function}) guaranties that for any $u\in\R$,
$\left(\gamma\left(u,\ell\right)\right)_{\ell\in\Z}$ is indeed the
autocovariance function of a stationary time series. 
Now, we can state the definition of  weakly locally stationary processes that we use
here.
\begin{definition}[Weakly locally stationary
  processes] \label{definition:locally_stationary} Let $(X_{t,T})_{t\in\Z,T\geq
    T_0}$ be an array of random variables with finite variances and
  $T_0,\beta,R>0$. We say that $(X_{t,T})_{t\in\Z,T\geq T_0}$ is
  $(\beta,R)$-\emph{weakly locally stationary with local spectral density $f$}
  if, for all $\lambda\in\R$, we have
  $f(\cdot,\lambda)\in\Lambda_{1}(\beta,R)$, and the time varying covariance
  function $\gamma^*$ of $(X_{t,T})_{t\in\Z,T\geq T_0}$ and the local
  covariance function $\gamma$ associated with $f$ satisfy, for all $t\in\zset$
  and $T\geq T_0$,
	\begin{eqnarray} \label{equation:gamma_gamma_star}
		\left|\gamma^{*}\left(t,T,\ell\right)-\gamma\left(\frac{t}{T},\ell\right)\right| \leq R\,T^{-\min(1,\beta)}\;.
	\end{eqnarray}
\end{definition}
Let us give some examples fulfilling this definition.
\begin{example} \label{example:locally_stationary_Dahlhaus}
	Locally stationary processes were introduced in a general fashion by \cite{Dahlhaus:1996}
	using  the spectral representation
	\begin{eqnarray}
		X_{t,T} = \int_{-\pi}^{\pi}\rme^{\rmi t\lambda}A_{t,T}^{0}\left(\lambda\right)\xi\left(\rmd\lambda\right)\;, \label{equation:locally_stationary_process}
	\end{eqnarray}
	where $\xi(\rmd\lambda)$ is the spectral representation of a white noise and
	$(A_{t,T}^{0})_{t\in\Z,T\in\N^*}$ is a collection of transfer functions such
	that there exist a constant $K$ and a (unique) $2\pi-$ periodic function
	$A:\R\times\R\rightarrow\C$ with $A(u,-\lambda)=\overline{A(u,\lambda)}$
	such that for all $T\geq1$,
	\begin{eqnarray} \label{equation:A_locally_stationary}
		\sup_{t\in\Z,\lambda\in\R}\left|A_{t,T}^{0}\left(\lambda\right)-A\left(\frac{t}{T},\lambda\right)\right| \leq \frac{K}{T}\;. \label{equation:control_T_locally_stationary}
	\end{eqnarray}
	Provided adequate smoothness assumptions on the time varying transfer
        function $A$, this class of locally stationary processes satisfies
        Definition~\ref{definition:locally_stationary} (see
        \cite[Section~1]{Dahlhaus-JNPS:1996}) for some $\beta\geq1$ and
        $f(u,\lambda)=\left|A(u,\lambda)\right|^2$. The case $\beta\in(0,1]$
        can be obtained by raising $T$ to the power $\beta$
        in~(\ref{equation:A_locally_stationary}).
\end{example}

\begin{example}[Time-varying Causal Bernoulli Shift (TVCBS)]\label{example:non-stationary_CBS}
	Let $\varphi:\R\times\R^{\N}\to\R$. Consider, for all $T\geq1$ and
        $t\in\zset$, a mapping $\varphi^0_{t,T}:\R^{\N}\to\R$ defining the
        random variables
	\begin{eqnarray}\label{eq:CBS-def}
		X_{t,T} = \varphi^0_{t,T}\left((\xi_{t-k})_{k\geq0}\right)\;,
	\end{eqnarray}
	where $(\xi_t)_{t\in\Z}$ are i.i.d. We assume that
        $\E[|\xi_{0}|^{2r}]<\infty$ for some $r\geq1$ and that there exist $\beta,K>0$
        and $(\psi_k)_{k\geq0}\in\ell^1_+(\N)$, 
        such that, for all $T\geq1$, $t\in\zset$, $u,u'\in\R$ and
        $\bm x\in\R^{\N}$,
	\begin{eqnarray}\label{eq:CBS-def-1}
          \left|\varphi^0_{t,T}\left(\bm
          x\right)\right| &\leq&
                                 K\left(1+\sum_{k=0}^{\infty}\psi_k\left|x_{k}\right|\right)^{r}\;,
          \\
\label{eq:CBS-def-2}
\left|\varphi^0_{t,T}\left(\bm
    x\right)-\varphi\left(\frac tT,\bm x\right)\right| &\leq& 
K\,T^{-\min(1,\beta)}\,\left(1+\sum_{k=0}^{\infty}\psi_k\left|x_{k}\right|\right)^{r}\;,
	\end{eqnarray}
        It is easy to see that $(X_{t,T})_{t\in\Z,T\geq T_0}$
        satisfies~(\ref{equation:gamma_gamma_star}) with a constant $R$ only
        depending on $K$, $\beta$, $(\psi_k)_{k\in\N}$ and
        $\E[|\xi_{0}|^{2r}]$, and with local covariance function
        $\gamma(u,\cdot)$ defined as the covariance function of the stationary
        causal Bernoulli shift process $(X_t(u))_{t\in\zset}$ defined by
        $X_t(u)=\varphi(u,\left((\xi_{t-k})_{k\geq0}\right)$. To obtain that
        $(X_{t,T})_{t\in\Z,T\geq T_0}$ is $(\beta,R)$-weakly locally
        stationary, it thus only remains to check that $(X_t(u))_{t\in\Z}$ admits a
        spectral density $f(u,\cdot)$ and that the resulting local spectral
        density satisfies $f(\cdot,\lambda)\in\Lambda_{1}(\beta,R)$ for all
        $\lambda\in\R$.
\end{example}

\begin{example}[TVAR($p$) model]\label{example:tvar}
  Under suitable assumptions, the TVAR process is a particular case both of
  Example~\ref{example:locally_stationary_Dahlhaus} (see~\cite[Theorem
  2.3]{Dahlhaus:1996}) and Example~\ref{example:non-stationary_CBS} (see
  Section~\ref{section:TVAR}). It is defined as the (unique) \emph{stable}
  solution of the recursive equation
  \begin{eqnarray} \label{equation:TVAR-p}
    X_{t,T} = \sum_{j=1}^{p} \theta_{j}\left(\frac{t}{T}\right) X_{t-j,T} + \sigma\left(\frac{t}{T}\right) \xi_t \;,
  \end{eqnarray}
  where $\btheta=[\theta_{1}\,\ldots\,\theta_{p}]':\R\to\R^{p}$ are the
  time varying autoregressive coefficients and $(\xi_t)_{t\in\Z}$ are
  i.i.d. centred and with variance $1$. This example is detailed in
  Section~\ref{section:TVAR}. 
\end{example}

\subsection{Statement of the problem}
Consider a weakly locally stationary  $(X_{t,T})_{t\in\Z,T\geq T_0}$, which
we assume to have mean zero for convenience. Let $d\in\N^{*}$. For each
$t=1,\ldots,T$, define the prediction vector of order $d$ by
\begin{eqnarray} \label{equation:btheta_star_definition}
	\btheta_{t,T}^{*} &=& \mathop{\arg\min}_{\btheta=[\theta_{1}\,\ldots\,\theta_{d}]'\in\R^{d}}\EE\left[\left(X_{t,T}-\sum_{k=1}^d \theta_{k}\;X_{t-k,T}\right)^{2}\right] \nonumber\\ 
	&=& \mathop{\arg\min}_{\btheta\in\R^{d}}\EE\left[\left(X_{t,T}-\btheta^{\prime}\mathbf{X}_{t-1,T}\right)^{2}\right] \;,
\end{eqnarray}
where $A'$ denotes the
transpose of matrix $A$ and
$\mathbf{X}_{s,T}=\left[X_{s,T}\,\,\ldots\,\,X_{s-(d-1),T}\right]'$.

Let $\Gamma_{t,T}^{*}$ be the time varying covariances matrix
$\Gamma_{t,T}^{*}=(\gamma^{*}(t-i,T,j-i);\,i,j=1,\ldots,d)$ where $\gamma^{*}$ is
the time varying covariance function as defined
in~(\ref{equation:time_varying_covariance_function}). Provided
that $\Gamma^{*}_{t,T}$ is non-singular, the
solution of~(\ref{equation:btheta_star_definition}) is given by
\begin{equation} \label{equation:btheta_star}
	\btheta_{t,T}^{*} =  \left(\Gamma^{*}_{t,T}\right)^{-1}\bm\gamma_{t,T}^{*}\;,
\end{equation}
where
$\bm\gamma_{t,T}^{*}=[\gamma^{*}(t,T,1)\;\ldots\;\gamma^{*}(t,T,d)]'$. 
Analogously to~(\ref{equation:btheta_star}), and with the aim of approximating
the local solution of the stationary Yule-Walker equations, we set
\begin{eqnarray}
	\btheta_u &=& \Gamma_u^{-1}\bm\gamma_u\;, \label{equation:definition_theta_Y-W} 
\end{eqnarray}
where $\bm\gamma_{u}=[\gamma(u,1)\;\ldots\;\gamma(u,d)]'$, $\Gamma_{u}$ is the
covariances matrix $\Gamma_{u}=(\gamma(u,i-j);\,i,j=1,\ldots,d)$ and $\gamma$
is the local covariance function as defined
in~(\ref{equation:local_covariance_function}).
To ensure the above matrices to be non-singular, we assume a lower bound on the
local spectral density, yielding the main assumption used on the model which depends on some
positive constants $\beta, R$ and $f_{-}$ and on a local spectral density $f$.  
\begin{hypothesis}{M}
\item\label{hypothesis:locally_stationary_with_f} The sequence
  $(X_{t,T})_{t\in\Z,T\geq T_0}$ is a $(\beta,R)$--weakly locally stationary process
  with local spectral density $f$ in the sense of Definition
  \ref{definition:locally_stationary}. Assume moreover that the spectral density $f$ satisfies
  $f(u,\lambda)\geq f_{-}$ for all $u,\lambda\in\R$.
\end{hypothesis}
The following lemma allows us to control the error of the approximation
of the optimal linear prediction coefficients $\btheta_{t,T}^{*}$ by the local
ones $\btheta_{t/T}$. Its proof is postponed to appendix~\ref{sec:proof-lemma-refl}
for convenience. 
\begin{lemma} \label{lemma:theta_star_theta_control}
Let $d\in\N^{*},\beta>0, R>0$ and $f_{-}>0$. Suppose that
Assumption~\ref{hypothesis:locally_stationary_with_f} holds. Then, there exist
two constants $C_{1}, T_{0}>0$ depending only on $d$, $\beta$, $R$ and $f_{-}$ such
that, for all $t\in\Z$ and $T\geq T_{0}$,
\begin{align} \label{equation:inequality_theta_star_theta}
\left\|\btheta_{t,T}^{*}-\btheta_{t/T}\right\| \leq C_{1}\,T^{-\min(1,\beta)} \;.
\end{align}
\end{lemma}

An estimator $\widehat{\btheta}$ of $\btheta$ is studied in
\cite{Dahlhaus_Giraitis:1998} for the model of
Example~\ref{example:locally_stationary_Dahlhaus}. In the following we improve
these results by deriving minimax rate properties of the estimator of
\cite{Dahlhaus_Giraitis:1998} and extensions of it in a more general
setting.

In the following, the problem that we are interested is to derive a minimax
rate estimator $\widetilde{\btheta}$ at a given smoothness index $\beta>0$,
which means that, for such a $\beta$, the estimation risk, say the quadratic
risk $\E[\|\widetilde{\btheta}_{t,T}-\btheta_{t,T}^{*}\|^{2}]$, can be bounded
uniformly over all processes
satisfying~\ref{hypothesis:locally_stationary_with_f} (among with additional
assumptions), and that the corresponding rate of convergence as $T\to\infty$
cannot be improved by any other estimator. The case $\beta\leq2$ is solved in
\cite{Moulines_Priouret_Roueff:2005} for the subclass of TVAR models.

\subsection{Minimax estimation for adaptive prediction} 
\label{sec:how-minim-estim-prediction}

Let $\widehat{X}_{d,t,T}^{*}$ denote the best linear predictor of order $d$ of
$X_{t,T}$, which as a consequence of~(\ref{equation:btheta_star_definition}),
reads
\begin{eqnarray*}
	\widehat{X}_{d,t,T}^{*}=\left(\btheta_{t,T}^{*}\right)^{\prime}\mathbf{X}_{t-1,T}\;,
\end{eqnarray*}
We denote by $\widehat{X}_{t,T}^{*}$ the best predictor of $X_{t,T}$ given its
past, that is, the conditional expectation 
\begin{eqnarray}
	\widehat{X}_{t,T}^{*}=\EE\left[X_{t,T}\left|X_{s,T},\,s\leq
t-1\right.\right]\;.
\end{eqnarray}
As explained before, the goal of this paper is to derive 
estimators, say $\widetilde{\btheta}_{t,T}\in\R^{d}$, of  
$\btheta_{t/T}$, which is a local approximation of $\btheta_{t,T}^{*}$. 
In this section, we assume that $\widetilde{\btheta}_{t,T}$ is a function of the past $X_{s,T}$, $s\leq
t-1$. Then
$\widetilde{\btheta}_{t,T}^{\prime}\mathbf{X}_{t-1,T}$ is a legitimate
predictor of $X_{t,T}$ and we have the following decomposition of the
corresponding prediction quadratic risk
\begin{eqnarray*} 
\EE\left[\left(X_{t,T}-\widetilde{\btheta}_{t,T}^{\prime}\mathbf{X}_{t-1,T}\right)^{2}\right] = \EE\left[\left(X_{t,T}-\widehat{X}_{t,T}^{*}\right)^{2}\right] 
+ \EE\left[\left(\widetilde{\btheta}_{t,T}^{\prime}\mathbf{X}_{t-1,T}-\widehat{X}_{t,T}^{*}\right)^{2}\right] \;.
\end{eqnarray*}
The first term is the minimal prediction error that one would achieve with the
conditional expectation (which requires the true distribution of the whole
process).  Furthermore, inserting
$\widehat{X}_{d,t,T}^{*}=\left(\btheta_{t,T}^{*}\right)^{\prime}\mathbf{X}_{t-1,T}$
and using the Minkowskii and Cauchy-Schwartz inequalities, 
the square root of the second term can be bounded as
\begin{multline*} 
\left(\EE\left[\left(\widetilde{\btheta}_{t,T}^{\prime}\mathbf{X}_{t-1,T}-\widehat{X}_{t,T}^{*}\right)^{2}\right]\right)^{1/2}
\leq \left(\EE\left[\left(\widehat{X}_{d,t,T}^{*}-\widehat{X}_{t,T}^{*}\right)^2\right]\right)^{1/2}\\
+\left(\EE\left[\left\|\mathbf{X}_{t-1,T}\right\|^4\right]\right)^{1/4} 
\,\left(\EE\left[\left\|\widetilde{\btheta}_{t,T}-\btheta_{t,T}^{*}\right\|^4\right]\right)^{1/4}
 \;.
\end{multline*}
The first term in the upper bound is due to the approximation of the best
predictor by the best linear predictor of order $d$ and can only be improved by
increasing $d$. Note that, in the case of the TVAR($p$) model with $p\leq d$,
this error term vanishes.  The quantity $\EE[\|\mathbf{X}_{t-1,T}\|^4]$ is
typically bounded by a constant independent of $(t,T)$ over the class of
processes under consideration. Hence, for a given $d$, the control of the
prediction risk boils down to the control of the estimation risk
$\EE[\|\widetilde{\btheta}_{t,T}-\btheta_{t,T}^{*}\|^4]$.

To do so, we can further decompose the loss as
\begin{eqnarray}\label{eq:decomp-error-estimation}
\left\|\widetilde{\btheta}_{t,T}-\btheta_{t,T}^{*}\right\|\leq
\left\|\widetilde{\btheta}_{t,T}-\btheta_{t/T}\right\|+
\left\|\btheta_{t/T}-\btheta_{t,T}^{*}\right\|\;.
\end{eqnarray}
Note that the second term is a deterministic error basically accounting for the
approximation precision of the non-stationary model by a stationary one, a
bound of which is provided in Lemma~\ref{lemma:theta_star_theta_control} stated
above.

As a result of the successive error developments above, our efforts in the
following focus on controlling the estimation risk
$\EE[\|\widetilde{\btheta}_{t,T}-\btheta_{t/T}\|^4]$ uniformly over a class of
weakly locally stationary processes with given smoothness index $\beta>0$.

\subsection{Tapered Yule-Walker estimate} \label{section:Tapered_Yule-Walker_estimate}

Following \cite{Dahlhaus_Giraitis:1998}, a local empirical covariance function
is defined as follows. It relies on a real data taper function $h$
and a bandwidth $M$ which may depend on $T$.

\begin{definition}[Local empirical covariance function] \label{definition:empirical_local_covariance_function}
	Consider a function $h:[0,1]\to\R$ and $M\in 2\N^{*}$. The empirical
	local covariance function $\widehat{\gamma}_{T,M}$ with taper $h$ is defined in $\R\times\Z$ as
	\begin{eqnarray*} \label{equation:local_empirical_covariance_function}
		\widehat{\gamma}_{T,M}\left(u,\ell\right) = \frac{1}{H_{M}}\sum_{\substack{t_{1},t_{2}=1\\t_{1}-t_{2}=\ell}}^{M}h\left(\frac{t_{1}}{M}\right)h\left(\frac{t_{2}}{M}\right)X_{\lfloor uT\rfloor+t_{1}-M/2,T}X_{\lfloor uT\rfloor+t_{2}-M/2,T}\;,
	\end{eqnarray*}
	where $H_{M}=\sum_{k=1}^{M}h^{2}(k/M)\sim M\int_{0}^{1}h^{2}(x)\rmd x$
        is the normalizing factor. If $H_{M}=0$, we set
        $\widehat{\gamma}_{T,M}\left(u,\ell\right)=0$, by convention.
\end{definition}
For $h\equiv 1$ in
Definition~\ref{definition:empirical_local_covariance_function} we obtain the
classical covariance estimate for a centred sample
$\{X_s, \lfloor uT\rfloor-M/2\leq s\leq \lfloor uT\rfloor+\ell+M/2\}$.  For any
$d\geq1$, based on the local empirical covariance function
$\widehat{\gamma}_{T,M}$, the $d$-order local empirical Yule-Walker prediction
coefficients are then defined as
\begin{eqnarray}
	\widehat{\btheta}_{t,T}\left(M\right) &=& \widehat{\Gamma}^{-1}_{t,T,M}\widehat{\bm\gamma}_{t,T,M}\;, \label{equation:definition_hat_theta_Y-W}
\end{eqnarray}
where
$\widehat{\bm\gamma}_{t,T,M}=[\widehat{\gamma}_{T,M}(t/T,1)\;\ldots\;\widehat{\gamma}_{T,M}(t/T,d)]'$,
$\widehat{\Gamma}_{t,T,M}$ is the matrix of empirical covariances
$\widehat{\Gamma}_{t,T,M}=(\widehat{\gamma}_{T,M}(t/T,i-j);\,i,j=1,\ldots,d)$. 
The only way $\widehat{\Gamma}_{t,T,M}$ can be singular is when
$\widehat{\gamma}_{T,M}\left(t/T,\ell\right)=0$ for all $\ell\in\Z$ (see
Lemma~\ref{lemma:hat_theta_Y-W_bound}), in which case we just set
$\widehat{\btheta}_{t,T}\left(M\right):=0$. Hence
$\widehat{\btheta}_{t,T}\left(M\right)$ is always well defined and always
satisfies the following  (see again Lemma~\ref{lemma:hat_theta_Y-W_bound} for
the bound)
\begin{align}\label{equation:definition_hat_theta_Y-W_allcases}
	\widehat{\Gamma}_{t,T,M}\widehat{\btheta}_{t,T}\left(M\right) =
                                      \widehat{\bm\gamma}_{t,T,M}\;, \quad\text{and}\quad
\|\widehat{\btheta}_{t,T}\left(M\right)\|\leq 2^d-1\;.
\end{align}
Using this trick, we do not find it necessary to add additional assumptions on
the model to guarantee that $\widehat{\Gamma}_{t,T,M}$ is non-singular a.s., as
done for instance in \cite{Dahlhaus_Giraitis:1998}, where $\PP(X_{t,T}=0)=0$
for all $t\in\zset$ is assumed.

\section{Main results in the general framework} \label{section:main_results}
\subsection{Additional notation and assumptions}
For convenience, we introduce the following notation. Let $p>0$,
$q,r,s\in\N^{*}$, $\mu$ be a probability distribution on $\R$, $u:\R\to\R$,
$a,b:\R^{r}\to\R$, $\bm{c}\in\R^{q}$ and a collection of random matrices
$\{U_{M}\in\R^{r\times s}, M\in\N^{*}\}$. We write
\begin{enumerate}[label=(\roman*)]
	\item $U_{M} = O_{L^{p},\bm{c}}(u(M))$ if there exists $C_{p,\bm{c}}>0$, depending continuously and at most on $(p,\bm{c}')$, such that for all $M\in\N^{*}$
	\begin{eqnarray} \label{equation:fast_decreasing_moments}
		\max_{1\leq i\leq r,1\leq j\leq s}\left(\E\left[\left|U_{Mi,,j}\right|^{p}\right]\right)^{1/p} \leq C_{p,\bm{c}}\left|u\left(M\right)\right|\;,
	\end{eqnarray}
	where $U_{M,i,j}$ is the $(i,j)$-th entry of the matrix $U_{M}$.
 	\item $U_{M} = O_{L^{\bullet}(\mu),\bm{c}}(u(M))$ if $U_{M} =
          O_{L^{p},m_p,\bm{c}}(u(M))$ for all $p\in[1,\infty)$, where $m_p$ is
          a constant only depending on the absolute moments of the distribution
          $\mu$, $\int|x|^q\,\mu(\rmd x)$, $q\geq1$.
	\item $a(\bm{x}) = O_{\bm{c}}(b(\bm{x}))$ if and only if there exists a constant $C_{\bm{c}}$ depending continuously and at most on the index $\bm{c}$, such that for all $\bm{x}\in\R^{r}$ 
	\begin{eqnarray*}
		\left|a\left(\bm{x}\right)\right| \leq C_{\bm{c}}\left|b\left(\bm{x}\right)\right|\;.
	\end{eqnarray*}
\end{enumerate}

Concerning the function $h$ we have the following assumption. 
\begin{hypothesis*}{H} 
\item\label{hypothesis:h} The function $h:[0,1]\to\R$ is piecewise continuously
  differentiable, that is, for $0=u_{0}<u_{1}<\ldots<u_{N}=1$, $h$ is
  $\mathcal{C}^{1}$ on $(u_{i-1},u_{i}]$, $i=1,\ldots,N$. Moreover we assume
  $\int_0^1h^2=1$,
  $\|h\|_{\infty}=\sup_{u\in[0,1]}|h(u)|<\infty$ and
  $\|h'\|_{\infty}=\max_{1\leq i\leq N}\sup_{u\in(u_{i-1},u_{i}]}|h'(u)|<\infty$.
\end{hypothesis*}

Provided a piecewise continuously differentiable funtion $h$ (as in
~\ref{hypothesis:h}) and a local spectral density function $f$ continuously differentiable on its
first argument, we also consider the following assumption, which depends on a
constant $C>0$ and on a probability distribution $\mu$ on $\R$.   
\begin{hypothesis*}{C}
\item\label{hypothesis:gamma} 
For all $\ell\in\Z$ and $h$ satisfying~\ref{hypothesis:h}, we have, for all $t\in\zset$ and
$T\geq T_0$,  
$$\widehat{\gamma}_{T,M}\left(u,\ell\right)-\E\left[\widehat{\gamma}_{T,M}\left(u,\ell\right)\right]
  = O_{L^{\bullet}(\mu),\ell,\|h\|_{\infty},\|h'\|_{\infty},C}\left(M^{-1/2}\right)\;.$$
\end{hypothesis*}
Assumption~\ref{hypothesis:gamma} amounts to say that the tapered empirical
covariance estimator $\widehat{\gamma}_{T,M}$ from a sample of length $M$
satisfies a standard deviation rate $M^{-1/2}$ in all $L^q$-norms.  Locally
stationary processes of Example~\ref{example:locally_stationary_Dahlhaus}
satisfy it under suitable assumptions (see~\cite[Eq.~(4.4) in
Theorem~4.1]{Dahlhaus_Giraitis:1998}).  We conclude this section with a result that can be used
for processes of Example~\ref{example:non-stationary_CBS}.
\begin{theorem}\label{thm-ass-C-for-CBS}
  Let $(X_{t,T})_{t\in\Z,T\geq T_0}$ be an array of random variables defined as
  in~(\ref{eq:CBS-def}) where $(\xi_t)_{t\in\Z}$ are i.i.d. satisfying
  $\E[|\xi_{0}|^{q}]<\infty$ for all $q\geq 1$ and $\varphi^0_{t,T}$
  satisfies~(\ref{eq:CBS-def-1}) for some $(\psi_k)_{k\in\N}\in\ell^1_+(\N)$,
  $K>0$ and $r\geq1$. Assume moreover that there exist $k_0\in\N$,
  $(\zeta_k)_{k\in\N}\in\ell^1_+(\N)$ such that for all $t\in\zset$, $T\geq T_0$
  and all $\bm x,\bm x'\in\R^\N$ satisfying $x_k=x'_k$ for $1\leq k\leq k_0$,
\begin{equation}
  \label{eq:lip-cond-CBS-rep}
\left|\varphi^0_{t,T}(\bm x)-\varphi^0_{t,T}(\bm x')\right|\leq K\,
\left(\sum_{k\geq0}\zeta_k|x_{k_0+k}-x'_{k_0+k}|\right)\,\left(1+\sum_{k\geq0}\psi_k(|x_k|+|x'_k|)\right)^{r-1}\;.
\end{equation}
Suppose moreover that 
\begin{equation}
  \label{eq:lip-cond-CBS}
\sum_{k\geq0}k\zeta_k<\infty\;.  
\end{equation}
Then there exists a constant $C$ only depending on $r,K,k_0,(\psi_k)_{k\in\N}$,
$(\zeta_k)_{k\in\N}$ and the distribution of $\xi_0$  such
that~\ref{hypothesis:gamma} holds.
\end{theorem}
The proof is postponed to 
Appendix~\ref{sec:proof-CBS-locally-stationary-proof}.

\subsection{Bound of the estimation risk}
Our first result on the estimation risk is a uniform approximation for the
estimation error of $\widehat{\btheta}_{t,T}(M)$.

\begin{theorem} \label{theorem:thetas_difference_development} Suppose that
  Assumption~\ref{hypothesis:locally_stationary_with_f} holds with some
  $\beta>0$, $f_{-}>0$ and $R>0$, and let $h:[0,1]\to\R$
  satisfying~\ref{hypothesis:h}.  Let $k\in\N$ and $\alpha\in(0,1]$ be uniquely
  defined by the relation $\beta=k+\alpha$. Suppose that
  Assumption~\ref{hypothesis:gamma} holds for some constant $C>0$ and
  distribution $\mu$.  Then, for any $d\geq1$, the estimator
  $\smash{\widehat{\btheta}_{t,T}(M)}$ defined
  by~(\ref{equation:definition_hat_theta_Y-W}) satisfies
	\begin{eqnarray} \label{equation:thetas_difference_development}
		\widehat{\btheta}_{t,T}\left(M\right)-\btheta_{t/T} = \sum_{\ell=1}^{k}\bm a_{h,f,\ell}\left(\frac{M}{T}\right)^{\ell}
		+ O_{d,f_{-},\|h\|_{\infty},\|h'\|_{\infty},\beta,R}\left(\frac{1}{M}+\left(\frac{M}{T}\right)^{\beta}\right) + \bm v_{M} \;,
	\end{eqnarray}
	where $\bm a_{h,f,\ell}\in\R^d$ depends only on $h$, $f$ and $\ell$ and
        $\bm v_{M} =
        O_{L^{\bullet}(\mu),d,f_{-},\|h\|_{\infty},\|h'\|_{\infty},\beta,R,C}(M^{-1/2})$. Moreover,
        if $h(x)=h(1-x)$ for $x\in[0,1]$, then  $\bm a_{h,f,1}=0$.
\end{theorem}

The proof is postponed to Appendix~\ref{section:proof_theorem_thetas_difference_development}. 
\begin{remark}
  In~(\ref{equation:thetas_difference_development}), the choice of
  the taper may influence the rate of convergence through the constant
  $\bm a_{h,f,1}$, which vanishes if the taper is symmetric, that is, if
  $h(x)=h(1-x)$ for $x\in[0,1]$. Other constants depend on the choice of the
  taper but one cannot choose tapers that ensure a further systematic
  improvement of the rate. The reason is given by the definition of the
  constant $c_{h,2}$ appearing in the proof of
  Theorem~\ref{theorem:expectation_J_M}, which implies $c_{h,2}>0$ for all
  tapers $h$. Consequently, for any taper $h$, one have that
  $\bm a_{h,f,2}\neq0$, except perhaps for some particular local density
  functions $f$. Hence, as far as rates of convergence are concerned, the only
  important property of the taper is that of being symmetric.
\end{remark}
Theorem~\ref{theorem:thetas_difference_development} suggests to combine several
$\widehat{\btheta}_{t,T}(M)$ to obtain a more accurate estimation by canceling
out the first $k$ bias terms
in~(\ref{equation:thetas_difference_development}). This technique was already
used for eliminate one term of bias in
\cite[Theorem~8]{Moulines_Priouret_Roueff:2005}. It is inspired by the
Romberg's method in numerical analysis (see
\cite{Baranger_Brezinski:1991}). Let
$\bm\omega=[\omega_{0}\;\ldots\;\omega_{k}]'\in\R^{k+1}$, be the solution of
the equation
\begin{eqnarray} \label{equation:alpha}
	A\bm\omega=\bm e_{1} \;, 
\end{eqnarray}
where $\bm e_{1}=[1\;0\;\dots\;0]'$ is the $\R^{k+1}$- vector having a $1$ in
the first position and zero everywhere else and $A$ is a $(k+1)\times(k+1)$
matrix with entries $A_{i,j}=2^{i\,j}$ for $0\leq i,j\leq k$.

\begin{theorem} \label{theorem:combined_thetas_difference_development} Under
  the same assumptions as Theorem~\ref{theorem:thetas_difference_development},
  the estimator
  \begin{equation}
    \label{eq:bias-reduced-estimator}
  \widetilde{\btheta}_{t,T}(M)=\sum_{j=0}^{k}\omega_{j}\widehat{\btheta}_{t,T}(2^{j}M)\;,    
  \end{equation}
  with $\bm\omega$ defined by~(\ref{equation:alpha}), satisfies
	\begin{eqnarray} \label{equation:combined_thetas_difference_development}
		\widetilde{\btheta}_{t,T}\left(M\right)-\btheta_{t/T} = O_{d,f_{-},\|h\|_{\infty},\|h'\|_{\infty},\beta,R}\left(\frac{1}{M}+\left(\frac{M}{T}\right)^{\beta}\right) +O_{L^{\bullet}(\mu),d,f_{-},\|h\|_{\infty},\|h'\|_{\infty},\beta,R,C}(M^{-1/2}) \;.
	\end{eqnarray}
\end{theorem}
The proof is postponed to
Appendix~\ref{section:proof_theorem_combined_thetas_difference_development}.
\begin{remark}
  If $h(x)=h(1-x)$ for $x\in[0,1]$ then the first order term
  of~(\ref{equation:thetas_difference_development}) is zero; in this case we
  can remove the term $j=k$ in~(\ref{eq:bias-reduced-estimator}) and define
  $\bm\omega=[\omega_{0}\;\ldots\;\omega_{k-1}]'\in\R^{k}$
  by~(\ref{equation:alpha}) with the second row and last column of $A$ removed
  and  $\bm e_{1}=[1\;0\;\dots\;0]'\in\R^{k}$.
\end{remark}
It is straightforward to check that the optimal bandwidth for minimizing the
order of the right term of
Equation~(\ref{equation:combined_thetas_difference_development}) is $M\propto
T^{2\beta/(2\beta+1)}$, yielding the next result.

\begin{corollary} \label{corollary:tilde_theta_star}
	Let $\beta, R, f_{-}>0$ and $h:[0,1]\to\R$. Let $k\in\N$ and
        $\omega\in(0,1]$ be uniquely defined such that $\beta=k+\omega$ and
        set $M:=2\lfloor T^{2\beta/(2\beta+1)}\rfloor$ in the following. Suppose that
        Assumptions~\ref{hypothesis:locally_stationary_with_f},~\ref{hypothesis:h}
        and~\ref{hypothesis:gamma} hold. Let
        $\smash{\widetilde{\btheta}_{t,T}(M)}$ be obtained as in
        Theorem~\ref{theorem:combined_thetas_difference_development}. Then, for
        any $q>0$ there exist a constant $C_0$ only depending on $h,q,d,f_{-},R,\mu$
        and $\beta$, and a $T_{0}>0$ only depending on $d,R,\beta$,
        $f_{-}$ and $C$ such that, for all $T\geq T_{0}$ and all $t\in\Z$,
	\begin{eqnarray} \label{equation:bound_theta_tilde_theta_star}
          \left(\E\left[\left\|\widetilde{\btheta}_{t,T}\left(M\right)-\btheta_{t/T}\right\|^{q}\right]\right)^{1/q}
          \leq C_0\,T^{-\beta/(2\beta+1)}\;.
	\end{eqnarray}
\end{corollary}
It is interesting to note that in the
decomposition~(\ref{eq:decomp-error-estimation}), the bound of the error term
$\|\widetilde{\btheta}_{t,T}\left(M\right)-\btheta_{t/T}\|$
in~(\ref{equation:bound_theta_tilde_theta_star}) always has a slower decaying rate
that that of the bound of the (deterministic) error term
$\|\btheta_{t/T}-\btheta_{t,T}^*\|$ in
Lemma~\ref{lemma:theta_star_theta_control}.

\section{Application to TVAR processes} \label{section:TVAR}

TVAR processes (see Example~\ref{example:tvar}) are a handful model to
illustrate our results. Under suitable assumptions, they have the specific
property that, when $d= p$, the linear predictor coefficients in
$\btheta_{t,T}^{*}\in\R^{d}$ as defined by
Equation~(\ref{equation:btheta_star_definition}) coincide with the time-varying
autoregressive coefficients given by $\btheta(t/T)$ of the TVAR($p$)
equation~(\ref{equation:TVAR-p}) and also with the local solution
$\btheta_{t/T}$ of the Yule-Walker equations defined
by~(\ref{equation:definition_theta_Y-W}),
see~(\ref{eq:pred-tvar-allcoeff-same}) below. 

In the following, we introduce some smoothness assumptions on the time-varying
parameters, similar to (and actually yielding) the one required on the local
spectral density in~\ref{definition:locally_stationary}. Additional stability
conditions are also required, based on the following definitions. For
$\bm\theta:\R\rightarrow\R^{p}$, $u\mapsto
\begin{bmatrix}\theta_{1}(u)&\dots&\theta_{p}(u)\end{bmatrix}'$
we define the time varying autoregressive
polynomial by
$$
\btheta(z;u) := 1-\sum_{j=1}^{p}\theta_{j}(u)z^{j}\;.  
$$
Let us define,  for any $p\in\N^*$ and $\delta>0$, 
\begin{align}
  \label{eq:definition:s_d-simple}
s_{(p)}\left(\delta\right)&:=\left\{\btheta=\begin{bmatrix}\theta_{1}&\dots&\theta_{p}\end{bmatrix}'\in\R^{p}\;\text{s.t.}\;1-\sum_{j=1}^{p}\theta_{j}z^j\neq 0, \forall
|z|<\delta^{-1}\right\}\;, \\ 
  \label{definition:s_d}
  s_{p}(\delta)&:=\left\{\btheta:\R\to s_{(p)}\left(\delta\right)\right\}\\
\nonumber
                          &=\{\btheta:\R\to\R^{p}\;\text{s.t.}\;\btheta(z;u)\neq 0, \forall
|z|<\delta^{-1}, u\in\R\}\;.
\end{align}
Define, for $\beta>0$, $R>0$, $\delta\in(0,1)$, $\rho\in[0,1]$ and $\sigma_+>0$,
the class of parameters
\begin{eqnarray*}
\mathcal{C}\left(\beta,R,\delta,\rho,\sigma_{+}\right)=\left\{(\boldsymbol{\theta},\sigma):\R\to\R^p\times[\rho\sigma_+,\sigma_+]\;\text{s.t.}\;\boldsymbol{\theta}\in\Lambda_p(\beta,R)\cap
s_p(\delta),\sigma\in\Lambda_1(\beta,R)\right\}\;.
\end{eqnarray*}
The first result to provide sufficient conditions on the TVAR coefficients for
the existence of a stable solution of the TVAR equations goes back to
\cite{Kunsch:1995}. Here we use
\cite[Proposition~1]{Giraud_Roueff_Sanchez-Perez:2015}, which guarantees the
following: given a centered i.i.d. sequence $(\xi_t)_{t\in\Z}$ with unit
variance and given the constants $\delta\in(0,1)$, $\rho\in[0,1]$, $\sigma_+>0$,
$\beta>0$ and $R>0$, there exists a large enough $T_0$ only depending on
$\delta,\beta$ and $R$ such that, for all
$(\btheta,\sigma)\in\mathcal{C}\left(\beta,R,\delta,\rho,\sigma_{+}\right)$,
there exists a unique process $(X_{t,T})_{t\in\Z,T\geq T_0}$
satisfying~(\ref{equation:TVAR-p}) for all $t\in\Z$ and $T\geq T_0$ and such
that, for all $T\geq T_0$, $X_{t,T_0}$ is bounded in probability as
$t\to-\infty$. We use this result as our definition of the TVAR process with
time varying AR coefficients
$\theta_1,\dots,\theta_p$, time varying standard deviation $\sigma$, and innovations $(\xi_{t})_{t\in\Z}$. For later
reference, we summarize this in the following assumption.
\begin{hypothesis}{M}
\item\label{hypothesis:tvar_minimax} Let $(\xi_t)_{t\in\Z}$ be an
  i.id. sequence with zero mean and unit variance. Assume that
  $(\btheta,\sigma)\in\mathcal{C}\left(\beta,R_0,\delta,\rho,\sigma_{+}\right)$
  with $\delta\in(0,1), \beta>0, R_0>0$ and $\rho\in[0,1]$. The array
  $(X_{t,T})_{t\in\Z,T\geq T_0}$ is a TVAR process as previously defined with
  time varying AR coefficients $\theta_1,\dots,\theta_p$, time varying standard
  deviation $\sigma$, and innovations $(\xi_{t})_{t\in\Z}$.
\end{hypothesis}
In this assumption the constant $T_0$ is set to have the existence and
uniqueness of the stable solution of the TVAR equation for all $T\geq T_0$. It
may change hereafter from line to line to guarantee additional properties of
the solution but always in a way where it depends at most on the constants
$\beta,R_0,\delta,\rho$ and $\sigma_{+}$.  

The following assumption can be used
to control the moments of any order of the TVAR process.
\begin{hypothesis*}{I}
	\item\label{hyp:innov-moment_minimax} For all $q>0$ the innovations  $(\xi_{t})_{t\in\Z}$
	satisfy $\E\left[\left|\xi_0\right|^q\right]<\infty$.
\end{hypothesis*}
Time varying autoregressive processes are well known to be locally stationary
under certain conditions on their parameters and moments, see
\cite[Theorem~2.3]{Dahlhaus:1996}. Adapting these results
to our context, we have the following.
\begin{theorem} \label{theorem:A_TVAR} 
  Assumption~\ref{hypothesis:tvar_minimax} implies the two following
  assertions.
  \begin{enumerate}[label=(\roman*)]
  \item\label{item:tvar-1} There
  exist constants $R$ and $T_0$ only depending on 
  $p,\delta,\sigma_+,\beta$ and $R_0$ such
  that $(X_{t,T})_{t\in\Z,T\geq T_0}$ is $(\beta,R)$-weakly locally stationary in the
  sense of Definition \ref{definition:locally_stationary} with local spectral
  density defined by
	\begin{eqnarray}
  \label{eq:tvar-local-dens}
		f\left(u,\lambda\right)=\frac{\sigma^{2}\left(u\right)}{2\pi}
\left|\btheta\left(\rme^{-\rmi\lambda}\,;\,u\right)\right|^{-2}\;.
	\end{eqnarray}
        Moreover, we have, for all
        $T\geq T_0$ and $t\in\Z$,
        \begin{equation}
          \label{eq:pred-tvar-allcoeff-same}
          \btheta(t/T)=\btheta_{t,T}^*=\btheta_{t/T}\;,
        \end{equation}
        where $\btheta_{t,T}^*$ and $\btheta_{t/T}$ are the optimal and local
        prediction coefficients respectively defined
        by~(\ref{equation:btheta_star_definition})
        and~(\ref{equation:definition_theta_Y-W}) in the case $d=p$.
      \item\label{item:tvar-1bis} If $\rho\in(0,1]$, then
        Assumption~\ref{hypothesis:locally_stationary_with_f} holds with the
        same $\beta$ and some constants $R$, $T_0$ and $f_->0$ only depending
        on $p,\delta,\rho,\sigma_+,\beta$ and $R_0$.
      \item\label{item:tvar-1ter} If $\PP(\xi_0=x)=0$ for all $x\in\R$, then
        $\PP(X_{t,T}=0)=0$ for all $t\in\Z$ and $T\geq T_0$.
      \item\label{item:tvar-2} If~\ref{hyp:innov-moment_minimax} holds, then $(X_{t,T})_{t\in\Z,T\geq T_0}$
  satisfies Assumption~\ref{hypothesis:gamma} with $C$ only depending on $R_0$,
  $\beta,\delta$, $\sigma_+$ and with $\mu$ defined as the distribution of
  $\xi_0$. 
  \end{enumerate}
\end{theorem}
The proof is postponed to Appendix~\ref{section:results_tvar}.
Theorem~\ref{theorem:A_TVAR} basically shows that the results of
Section~\ref{section:main_results} apply to TVAR processes, as defined
by~\ref{hypothesis:tvar_minimax} provided that $\rho>0$
and~\ref{hyp:innov-moment_minimax} is assumed on the innovations. 
We specifically state the following result which provides a useful complement
to \cite[Corollary~9]{Moulines_Priouret_Roueff:2005} where the same minimax
rate is exhibited for a different estimator but only for smoothness index
$\beta\leq2$.

\begin{corollary} \label{corollary:tilde_theta_TVAR} Let
  $\delta\in(0,1), \beta>0, R>0$ and $\rho\in(0,1]$.  Suppose that
  Assumptions~\ref{hypothesis:tvar_minimax} and~\ref{hyp:innov-moment_minimax}
  hold. Let $M=2\lfloor T^{2\beta/(2\beta+1)}\rfloor$ and
  $\smash{\widetilde{\btheta}_{t,T}(M)}$ be the estimator defined
  by~(\ref{equation:definition_hat_theta_Y-W})
  and~(\ref{eq:bias-reduced-estimator}) with $p$, the order of the TVAR process
  equal to $d$, the order of the prediction vector. Then, for any
  $q\in\N$ there exists a constant $C$ only depending on
  $q, h, p, \delta, \rho, \sigma_{+},\beta, R_0$ and the moments of the distribution
  of $\xi_0$ such that, for all $T\geq T_{0}$ and $t\in\Z$, we have
	\begin{eqnarray}
		\left(\E\left[\left\|\widetilde{\btheta}_{t,T}\left(M\right)-\btheta\left(\frac{t}{T}\right)\right\|^{q}\right]\right)^{1/q} \leq C\,T^{-\beta/(2\beta+1)}\;.
	\end{eqnarray}
\end{corollary}
\begin{proof}
By Theorem~\ref{theorem:A_TVAR}, we can apply
Corollary~\ref{corollary:tilde_theta_star}. Recall that in the case of the TVAR with order
$p$ equal to the prediction length, we have
$\btheta\left(t/T\right)=\btheta_{t/T}=\btheta_{t,T}^*$. 
\end{proof}
The estimator $\widetilde{\btheta}$ proposed in
Corollary~\ref{corollary:tilde_theta_TVAR} achieves the $\beta$-minimax-rate for TVAR
processes according to the lower 
\cite[Theorem~4]{Moulines_Priouret_Roueff:2005}. Hence, it is also
$\beta$-minimax-rate in the class of weakly locally stationary processes satisfying
Assumption~\ref{hypothesis:locally_stationary_with_f}. \cite[Section~A.1]{Giraud_Roueff_Sanchez-Perez:2015}
explains how to construct minimax-rate predictors from minimax-rate
estimators of $\btheta$. Applying their approach,
Corollary~\ref{corollary:tilde_theta_TVAR} also provides a crucial ingredient
in building $\beta$-minimax-rate predictors for any $\beta>0$.

\section{Numerical work} \label{section:numerical_work_minimax}
We test both methods on data simulated according to a TVAR process with
$p=3$. The parameter function $u\mapsto\btheta(u)$ within $s_p(\delta)$ for
some $\delta\in(0,1)$ is chosen randomly as follows. First we pick randomly
some smoothly time varying partial autocorrelation functions up to the order
$p$ that are bounded between $-1$ and $1$,
$\check\theta_{k,k}\left(u\right) \propto
\sum_{j=1}^{F-1}a_{j,k}j^{2}\cos\left(ju\right)$, where $a_{j,k}$ are random
numbers in $[-1,1]$. Here $\check\theta_{k,k}\left(u\right)$ is defined up to a
multiplicative constant; dividing, for example, by $F(F-1)(2F-1)/6$ guarantees
its values to remain within $(-1,1)$. Then, for any required $t,T$, we use
Algorithm~\ref{algorithm:levinson-durbin} with input
$\check\theta_{k,k}\left(t/T\right)$ and assign the output to
$\btheta(t/T)$. Based on the classical Levinson-Durbin recurrence (see for
example \cite[Proposition~5.2.1]{Brockwell_Davis:2006}), the $\check\btheta$ in
Algorithm~\ref{algorithm:levinson-durbin} is in $s_{(p)}(1)$ as defined
in~(\ref{eq:definition:s_d-simple}), and it follows that the output
$\btheta\in s_{(p)}(\delta)$. The randomly obtained three components of our
$\btheta(t)$ are displayed in Figure~\ref{figure:thetas}, for $t\in[0,1]$.

For each $T\in\{2^{2j}, j=5,\ldots,15\}$ we generate $100$ realizations of a
TVAR process from innovation sequences $(\xi_t)_{t\in\Z}$ of i.i.d. centred
Gaussian random variables with unit variance by sampling the previous $\btheta$
at a rate $T^{-1}$, and taking $\sigma\equiv1$. 

Then we compare
$\smash{\widehat{\btheta}}$ and $\smash{\widetilde{\btheta}}$ for estimating
$\btheta(1/2)$ using $h\equiv 1$ and different values of $M$. Recall that
$\smash{\btheta(1/2)=\btheta_{T/2,T}^{*}}$. Figure~\ref{figure:boxplots_all_M}
shows the boxplots corresponding to this evaluation for two different $T$s.
In Figure~\ref{figure:boxplots_all_M} we observe that for $T=2^{20}$ the
estimation error of $\smash{\widehat{\btheta}}$ is minimized in $M=2^{15}$
while that of $\smash{\widetilde{\btheta}}$ is minimized in $M=2^{17}$. The
estimator $\smash{\widetilde{\btheta}}$ beats $\smash{\widehat{\btheta}}$ for
the two biggest values of $M$. In the case $T=2^{30}$, the error of
$\smash{\widehat{\btheta}}$ reaches its minimum in $M=2^{23}$ and that of
$\smash{\widetilde{\btheta}}$ in $M=2^{26}$. The estimator
$\smash{\widetilde{\btheta}}$ beats $\smash{\widehat{\btheta}}$ for the four
biggest values of $M$. These experiences illustrate the theoretical result
established in Theorem~\ref{theorem:thetas_difference_development} and
Corollary~\ref{corollary:tilde_theta_TVAR} that after optimizing in $M$,
$\smash{\widetilde{\btheta}}$ outperforms $\smash{\widehat{\btheta}}$ for
$T$ large enough.

To corroborate these conclusions over a wider range of $T$'s, we refer to
Figure~\ref{figure:oracle}. The plot on the left-hand side shows the oracle
errors $\min_{M}\|\widehat{\btheta}_{T/2,T}(M)-\btheta(1/2)\|$ and
$\min_{M}\|\widetilde{\btheta}_{T/2,T}(M)-\btheta(1/2)\|$ for all
$T\in\{2^{2j}, 5\leq j\leq15\}$. The slope corresponding to
$\smash{\widetilde{\btheta}}$ (in blue) is steeper than the one corresponding
to $\smash{\widehat{\btheta}}$ (in red), meaning that, in average,
$\smash{\widetilde{\btheta}}$ outperforms $\smash{\widehat{\btheta}}$ by an
increasing order of magnitude as $T$ increases. The boxplots on the right-hand
side of Figure~\ref{figure:oracle} represent the ratios
$\smash{\min_{M}\|\widetilde{\btheta}_{T/2,T}(M)-\btheta(1/2)\|/\min_{M}\|\widehat{\btheta}_{T/2,T}(M)-\btheta(1/2)\|}$
computed for each $T$ and realization of the TVAR process. Observe that for
$2^{14}\leq T\leq 2^{18}$ the estimator $\smash{\widetilde{\btheta}}$ beats
$\smash{\widehat{\btheta}}$ in at least half of the cases. For $T\geq 2^{20}$,
it happens in at least $75\%$ of the cases. We conclude that the estimator with
reduced bias is of interest when the length of the data set becomes very large. 

\begin{algorithm} 
	\caption{Adapted Levinson-Durbin algorithm.\label{algorithm:levinson-durbin}}
	\KwParameters{the stability parameter $\delta>0$ and the time varying partial
		autocorrelation functions $\check\theta_{k,k}$, $k=1,\ldots,p$}
	\For{$k=2$ \KwTo $p$}{
		\For{$j=1$ \KwTo $k-1$}{
			$\check\theta_{j,k} = \check\theta_{j,k-1}-\check\theta_{k,k}\check\theta_{k-j,k-1} $;
		}
	}
        \For{$j=1$ \KwTo $p$}{
			$\theta_{j,p} = \delta^j\check\theta_{j,p}$;
		}
	\Return{$\btheta=[\theta_{1,p}\,\ldots\, \theta_{p,p}]'$.}
\end{algorithm}
\begin{figure}[htbp]
	\centering
	\includegraphics[width=10cm]{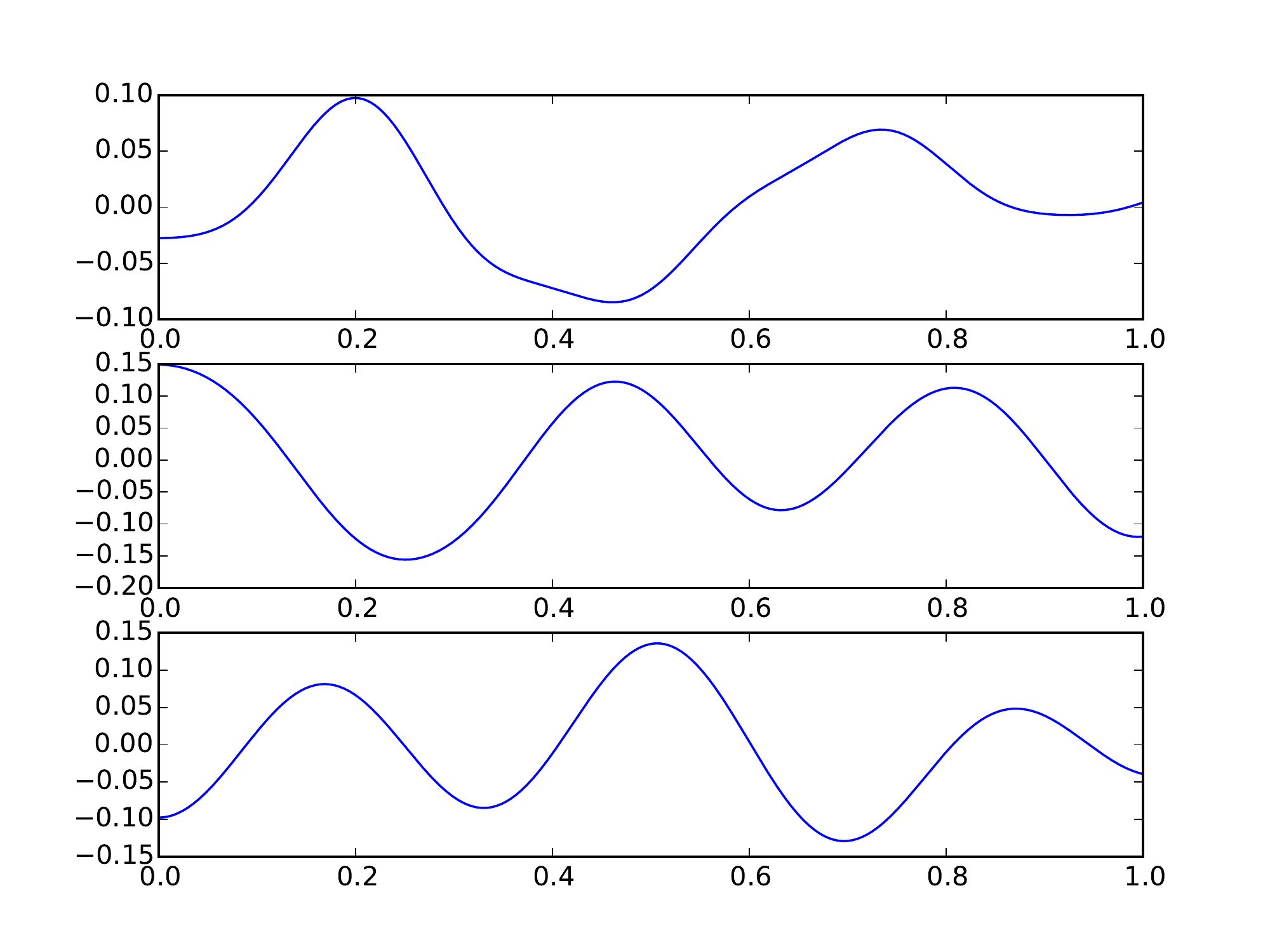}
	\caption{Plots of $\theta_{1}(t)$ (top),
		$\theta_{2}(t)$ (middle) and $\theta_{3}(t)$ (bottom) on the interval $t\in[0,1]$. } \label{figure:thetas}
\end{figure}

\begin{figure}[!h]
	\begin{minipage}[t]{0.4\linewidth}
		\centering
		\includegraphics[width=5.75cm]{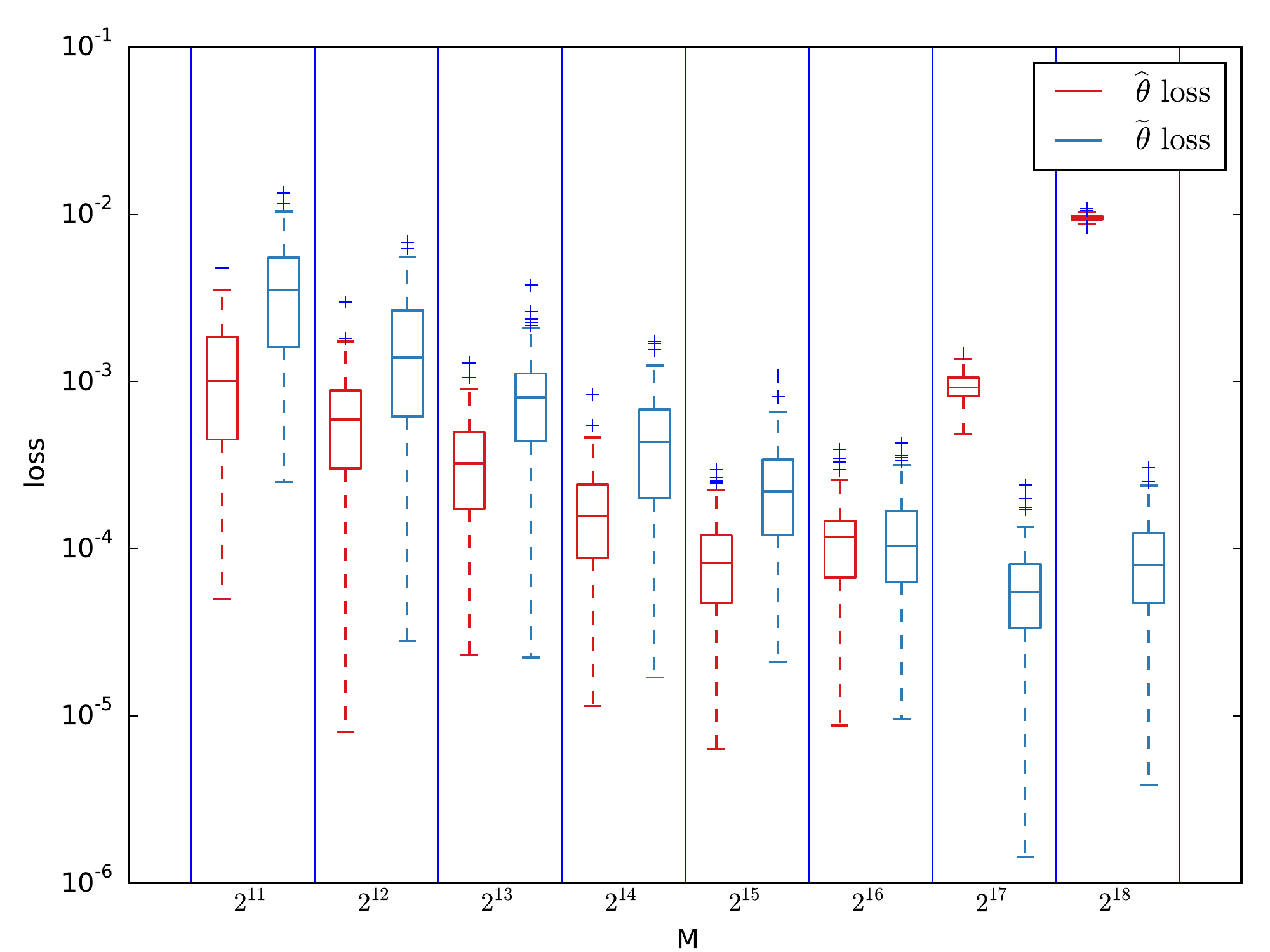}
	\end{minipage}
	\hspace{1.2cm}
	\begin{minipage}[t]{0.4\linewidth}
		\centering
		\includegraphics[width=5.75cm]{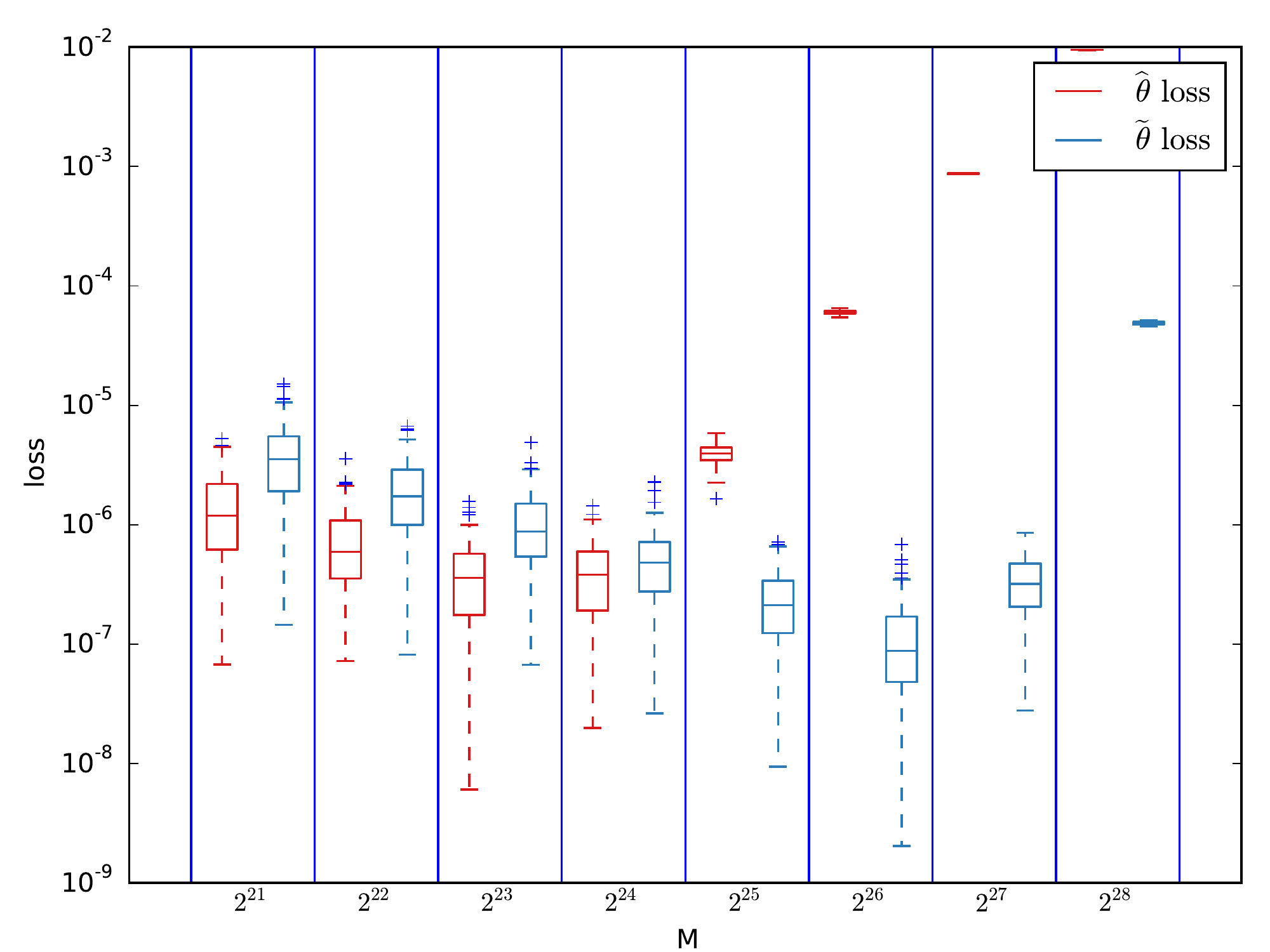}
	\end{minipage}
	\caption{Box plots of the quadratic losses for estimating $\btheta(1/2)$ using
		$\smash{\widehat{\btheta}_{T/2,T}(M)}$ (red boxes) and
		$\smash{\widetilde{\btheta}_{T/2,T}(M)}$ (blue boxes) for various
		bandwidths $M$, when  $T=2^{20}$ (left) and $T=2^{30}$ (right).}
	\label{figure:boxplots_all_M}
\end{figure}
\begin{figure}[!h]
	\begin{minipage}[t]{0.4\linewidth}
		\centering
		\includegraphics[width=5.75cm]{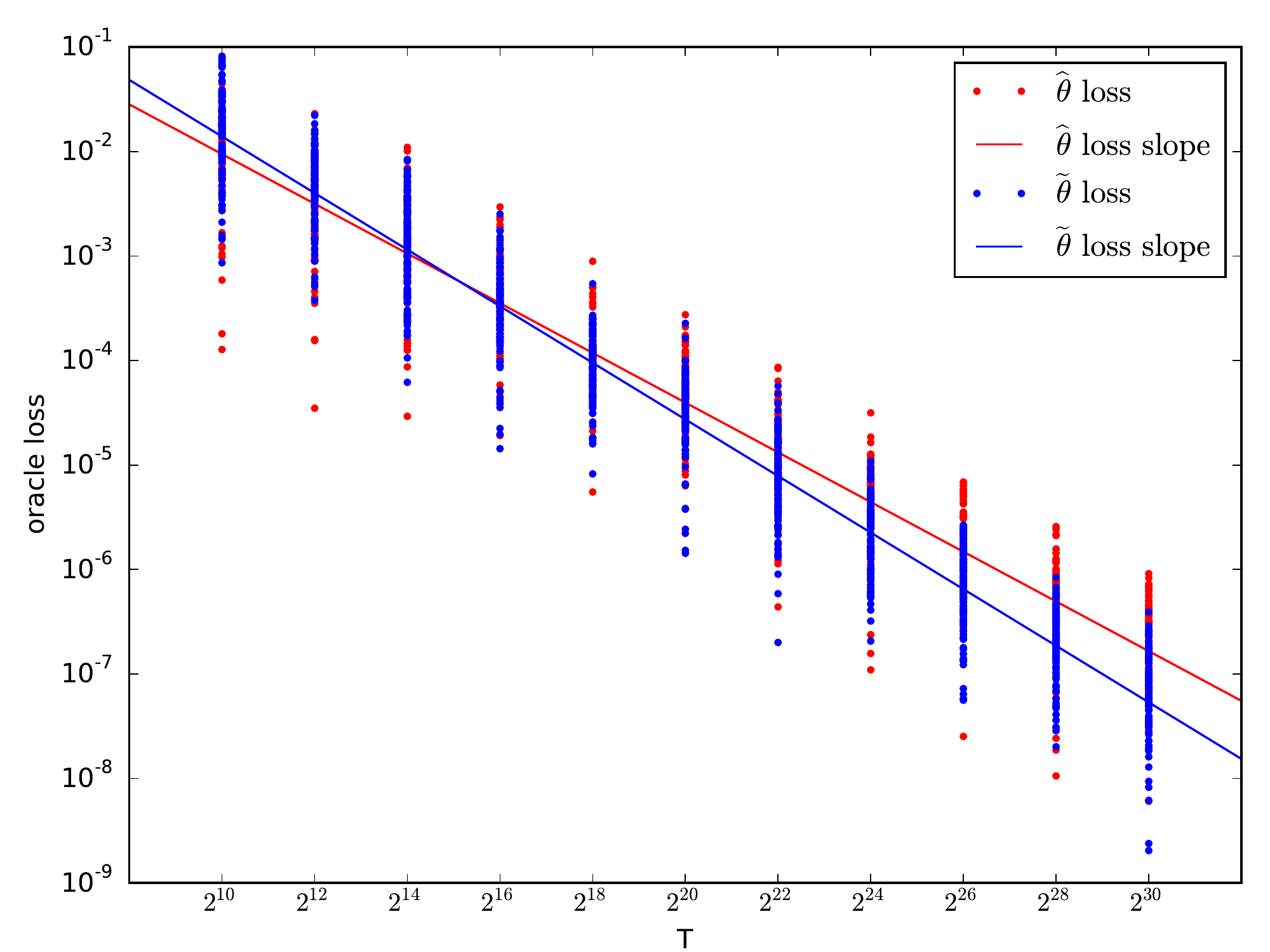}
	\end{minipage}
	\hspace{1.2cm}
	\begin{minipage}[t]{0.4\linewidth}
		\centering
		\includegraphics[width=5.75cm]{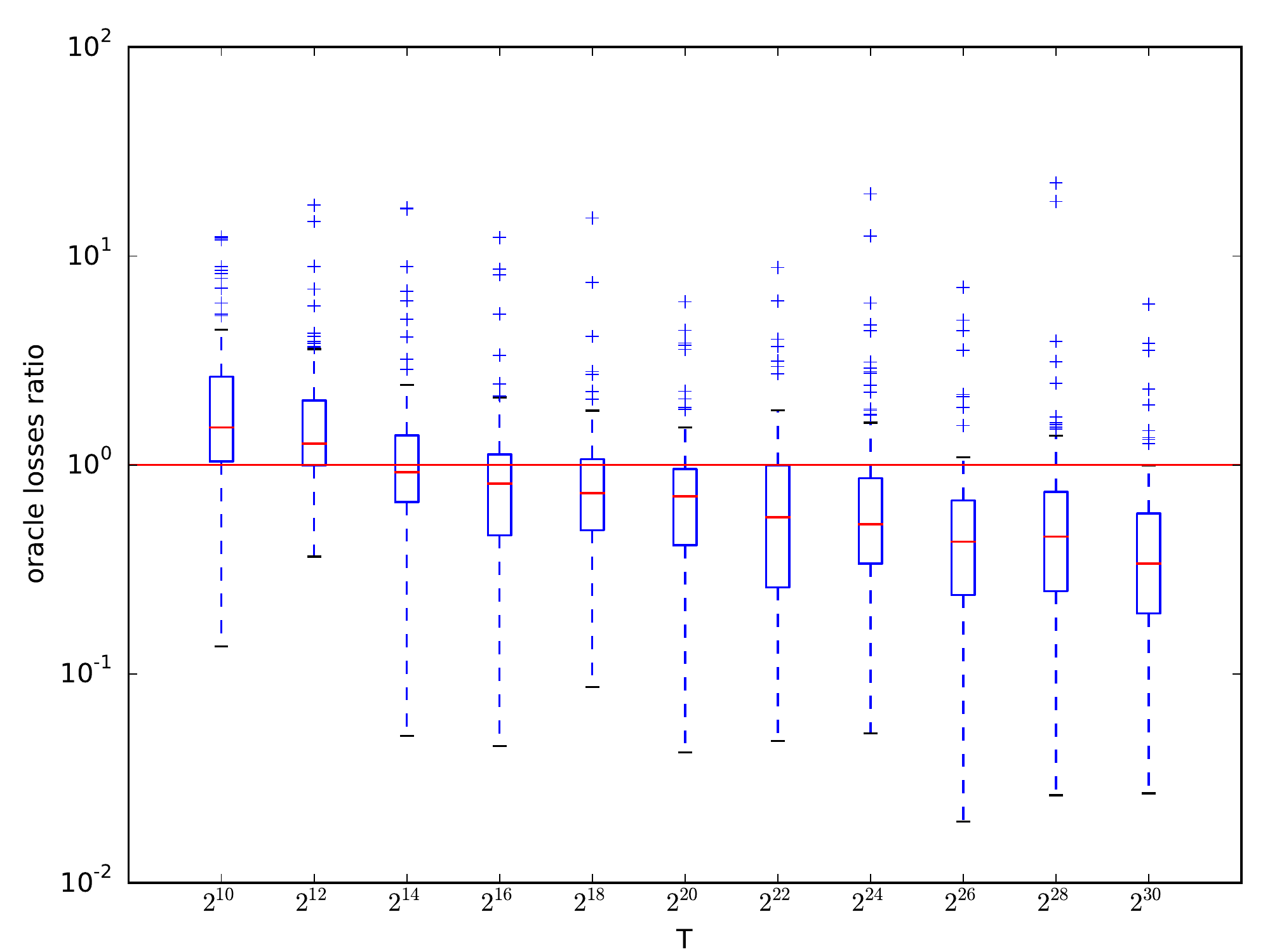}
	\end{minipage}
	\caption{Oracle losses (using the best choice for the bandwidth $M$) for estimating $\btheta(1/2)$ using
		$\smash{\widehat{\btheta}_{T/2,T}(M)}$ (red points) and
		$\smash{\widetilde{\btheta}_{T/2,T}(M)}$ (blue points) for various
		values of $T$. The left-hand side plot
		displays the losses over all the Monte Carlo simulations and the two resulting
		log-log regression lines. The right-hand side plot displays boxplots of the corresponding losses
		ratio. } \label{figure:oracle}	
\end{figure}

\section*{Acknowledgments}
We thank Tobias Kley for his thoughtful reading of the paper and appropriate
remarks (concerning specially Lemma~\ref{lemma:theta_star_theta_control}) and
also for bringing forward to our attention the related work \cite{kley16}.

\appendix

\section{Useful lemmas} \label{section:appendix_weakly_stationary_processes} We
gather here some useful lemmas that are (simple extensions of) standard results
for Yule-Walker estimation of the linear prediction coefficients. Most of them
are similar to those used in \cite{Dahlhaus_Giraitis:1998}.  Short proofs are
provided for the sake of completeness. Different bounds can be found in
\cite{kley16}, in order to better control the case $d\to\infty$. 

\begin{lemma} \label{lemma:thetas_difference_development} Let $d$ be a positive
  integer. Consider the $d\times d$ matrices $\Gamma$ and
  $\widehat{\Gamma}$ and vectors
  $\bm\gamma,\widehat{\bm\gamma},\bm\theta,\widehat{\bm\theta}\in\R^{d}$
  satisfying the relations
\begin{align}
\Gamma\bm\theta &= \gamma\;, \label{equation:Gamma_theta_gamma} \\
\widehat{\Gamma}\,\widehat{\bm\theta} &= \widehat{\gamma} \label{equation:hat_Gamma_theta_gamma} \;.
\end{align}
Then, for any $k\in\N$ we have, if $\Gamma$ is non-singular,
\begin{multline} \label{equation:thetas_difference_development_k}
\widehat{\btheta}-\btheta = \left(\Gamma^{-1}+\sum_{\ell=1}^{k}\left(\Gamma^{-1}\left(\Gamma-\widehat{\Gamma}\right)\right)^{\ell}\right)\left(\widehat{\bm \gamma}-\bm \gamma\right)+\sum_{\ell=1}^{k+1}\left(\Gamma^{-1}\left(\Gamma-\widehat{\Gamma}\right)\right)^{\ell}\bm\theta\\ + \left(\Gamma^{-1}\left(\Gamma-\widehat{\Gamma}\right)\right)^{k+1}\left(\widehat{\btheta}-\btheta\right) \;.
\end{multline}
\end{lemma}

\begin{proof}
From Equations~(\ref{equation:Gamma_theta_gamma}) and~(\ref{equation:hat_Gamma_theta_gamma}) we get 
\begin{eqnarray*}
		\widehat{\btheta}-\btheta &=& \Gamma^{-1}\left[\left(\Gamma-\widehat{\Gamma}\right)\widehat{\btheta}+\widehat{\bm \gamma}-\bm \gamma\right] \;. 
\end{eqnarray*}
The result follows by applying a recursion on $k=
0,1,\dots$.
\end{proof}

\begin{lemma} \label{lemma:eigenvalues_covariance_matrix-spectral_density} Let
  $\gamma$ be the autocovariance function associated with a spectral density
  function $f$,
  $\gamma\left(s\right) = \int_0^{2\pi}\rme^{\rmi
    s\lambda}\,f(\lambda)\;\rmd\lambda\;,$ for all $s\in\Z$, and denote by
  $\Gamma_d$ the corresponding covariance matrix of size $d\times d$,
\begin{eqnarray} \label{equation:covariance_matrix}
\Gamma_{d} = 
\begin{bmatrix}
\gamma\left(0\right) & \ldots & \gamma\left(d-1\right) \\
\vdots & \ddots & \vdots \\
\gamma\left(d-1\right) & \ldots & \gamma\left(0\right)
\end{bmatrix}\;.
\end{eqnarray}
Then the following assertions hold for any $d\in\N^{*}$.
        \begin{enumerate}[label=(\roman*)]
        \item\label{item:1} If  $\int_0^{2\pi} f>0$ then $\Gamma_d$ is positive definite. 
        \item\label{item:2} If  $f$ is valued in $[f_{-},f_{+}]$ with $f_{-}\leq f_{+}$, then all the
          eigenvalues of $\Gamma_d$ belong to $[2\pi f_{-},2\pi f_{+}]$.
        \end{enumerate}
\end{lemma}
\begin{proof}
  These well known facts  (see \emph{e.g.}
  \cite[Proposition~4.5.3]{Brockwell_Davis:2006}) follow from the identity $\displaystyle \bm
  a'\Gamma_{d}\bm a =
  \int_{-\pi}^{\pi}\left|\sum_{j=1}^{d}a_{j}\, \rme^{\rmi
      j\lambda}\right|^{2} f\left(\lambda\right)\rmd\lambda$ , for all $\bm
  a=[a_{1}\,\ldots\,a_{d}]'\in\R^{d}$.
\end{proof}
The next lemmas allow us to control the norms of $\widehat{\btheta}_{t,T}$ and
$\btheta_{t/T}$.
\begin{lemma}\label{lem:basic-s_pdelta}
  Let $p$ be a positive integer and $\delta>0$. The set
  $s_{(p)}\left(\delta\right)$ defined in~(\ref{eq:definition:s_d-simple}) is a
  closed subset of the ball
  $\{\btheta\in\R^p\;\text{s.t.}\;\|\btheta\|\leq(1+\delta)^{p}-1\}$.
\end{lemma}
\begin{proof}
Hurwitz's theorem
(see~\cite[Theorem~2.5]{Conway:1973} or \cite[Section~3,
Chapter~VIII]{Gamelin:2001}) implies that $s_{(p)}\left(\delta\right)$ is a
closed subset of $\R^p$. It is also bounded (see
\cite[Lemma~1]{Moulines_Priouret_Roueff:2005}). Hereafter we provide
a slightly different bound using Euclidean norm instead of the supnorm. 

Take now $\btheta\in s_{(p)}\left(\delta\right)$.  Let $z_{1},\ldots,z_{p}$
denote the complex roots of the polynomial
$\btheta(z):=1-\sum_{j=1}^{p}\theta_j z^{j}$. They satisfy
$|z_{j}|\geq\delta^{-1}$ for any $j$. The following holds
\begin{eqnarray} \label{equation:modulus_theta_hat}
1+\left\|\btheta\right\|^{2} = \frac{1}{2\pi}
  \int_{-\pi}^{\pi}\left|1-\sum_{j=1}^{p}\theta_{j}\rme^{\rmi
  j\lambda}
\right|^{2}\rmd\lambda = \frac{1}{2\pi} \int_{-\pi}^{\pi}\left|\btheta\left(\rme^{\rmi\lambda}\right)\right|^{2}\rmd\lambda \;.
\end{eqnarray}
On the other hand we have
$\btheta\left(z\right) = \prod_{j=1}^{p}\left(1-zz_{j}^{-1}\right)$, so that
for $|z|=1$, since $|z_{j}|^{-1}\leq\delta$, we get
$|\btheta\left(z\right)|\leq (1+\delta)^{p}$. Putting this
into~(\ref{equation:modulus_theta_hat}) the proof is completed.
\end{proof}
The next lemma is similar in flavor to the statistical result of
\cite[Section~3]{Whittle:1963}. It is also a classical property of orthogonal
polynomials (see \cite[Section~2.4]{Grenander_Szego:1984}). We provide an
elementary proof.
\begin{lemma} \label{lemma:roots_characteristic_polynomial_YW} Let $\gamma$ be
  an autocovariance function. Let $d\geq1$ such that the covariance matrix
  $\Gamma_{d}$ defined by Equation~(\ref{equation:covariance_matrix}) is
  positive-definite. Let $\btheta$ denote the solution of the $d$-order
  Yule-Walker equation,
  ${\btheta}=[{\theta}_{1}\,\ldots\,{\theta}_{d}]'=\Gamma_{d}^{-1}\bm\gamma_{d}$
  with $\bm\gamma_{d} = [\gamma(1)\,\ldots\,\gamma(d)]'$. Then we have
  $\btheta\in s_{(d)}(1)$ and $\|\btheta\|\leq2^d-1$.
\end{lemma}
\begin{proof}
  We only need to prove $\btheta\in s_{(d)}(1)$ since  $\|\btheta\|\leq2^d-1$
  is then implied by Lemma~\ref{lem:basic-s_pdelta} with $p=d$ and $\delta=1$.

  For $j=1,\ldots,d$, let $\bm e_{j}=[0\;\dots\;1\;\ldots\;0]'$ be the
  $\R^{d}$- vector having a $1$ in the $j$-th position and zero everywhere
  else. Consider also the companion matrix $A= \begin{bmatrix} {\btheta} & \bm
    e_{1} & \dots & \bm e_{d-1}
	\end{bmatrix}'$\;.
	Since the roots of ${\btheta}(z)$ are the inverses of the eigenvalues
        of $A$, or $A'$, we only need to prove that the eigenvalues of $A'$ are
        inside the closed unit disk. Observe that
	\begin{eqnarray*}
		\Gamma_{d}-A\Gamma_{d}A' = \Gamma_{d} - \begin{bmatrix}
			{\btheta}'\Gamma_{d}{\btheta} &{\btheta}'\Gamma_{d}\bm e_{1}&\dots&\dots &{\btheta}'\Gamma_{d}\bm e_{d-1} \\
			\bm e_{1}'\Gamma_{d}{\btheta} &\bm e_{1}'\Gamma_{d}\bm e_{1}&\dots&\dots&\bm e_{1}'\Gamma_{d}\bm e_{d-1} \\
			\bm e_{2}'\Gamma_{d}{\btheta} &\bm e_{2}'\Gamma_{d}\bm e_{1}&\dots&\dots&\bm e_{2}'\Gamma_{d}\bm e_{d-1} \\
			\vdots&\vdots&\ddots&\ddots&\vdots\\
			\bm e_{d-1}'\Gamma_{d}{\btheta}&\bm e_{d-1}'\Gamma_{d}\bm e_{1}&\dots&\dots&\bm e_{d-1}'\Gamma_{d}\bm e_{d-1}
		\end{bmatrix}\;.
	\end{eqnarray*}
	Because $\Gamma_{d}$ is a Toeplitz matrix, its $(i,j)$-th entries, and
        those of $A\Gamma_{d}A'$ are equal for $i,j\geq 2$. The definition of
        ${\btheta}$ implies also the equality of the $(i,j)$-th entries of both
        matrices when $i=1, j\geq 2$ and $i\geq 2, j=1$. 
        Hence $\Gamma_{d}-A\Gamma_{d}A'$ is a $d\times d$ symmetric matrix with
        zero entries except at the top-left where it takes value
        $\gamma(0)-{\btheta}'\bm\gamma_{d}$. This value is non-negative since
        it is the variance of the prediction error or order $d$. 
        Hence we conclude that for $\bm v\in\C^d$, $\bm
        v'(\Gamma_{d}-A\Gamma_{d}A')\bm v\geq0$. 
        Consider now $\lambda$, an eigenvalue of $A'$ and the corresponding
        eigenvector $\bm v\in\C^d\setminus\{0\}$. We get
$$
0\leq  \bar{\bm v}'(\Gamma_{d}-A\Gamma_{d}A')\bm v= \bar{\bm v}'(\Gamma_{d}-A\Gamma_{d}A')\bm v
= \bar{\bm v}'\Gamma_{d}\bm v (1-|\lambda|^2)\;.
$$ 
We conclude that $|\lambda|\leq1$ since $\bar{\bm v}'\Gamma_{d}\bm v>0$.
\end{proof}

\begin{lemma} \label{lemma:hat_theta_Y-W_bound} Let $d\geq1$,
  $(X_{t,T})_{t\in\Z,T\geq T_0}$ be an array process and $h:[0,1]\to\R$. For
  any $M\in\N^*$, define the local tapered empirical covariance function
  $\widehat{\gamma}_{T,M}$ as in
  Definition~\ref{definition:empirical_local_covariance_function} and let, for
  any $t\in\Z$ and $T\geq T_0$,
  $\widehat{\Gamma}_{t,T,M}=(\widehat{\gamma}_{T,M}(t/T,i-j);\,i,j=1,\ldots,d)$
  be the corresponding $d\times d$ empirical covariance matrix.  Then, either
  $\widehat{\Gamma}_{t,T,M}$ is non-singular, or
  $\widehat{\gamma}_{T,M}(t/T,\ell)=0$ for all $\ell\in\Z$. Moreover, in the
  case where it is non-singular, the Yule-Walker estimate
  $\widehat{\btheta}_{t,T}(M)$ defined
  by~(\ref{equation:definition_hat_theta_Y-W}) satisfies
  $\|\widehat{\btheta}_{t,T}(M)\|\leq 2^{d}-1$.
\end{lemma}
\begin{proof}
  First note that for all $u\in\R$, the sequence
  $\left(\widehat{\gamma}_{T,M}\left(u,\ell\right)\right)_{\ell\in\Z}$ is the
  covariance function associated with the spectral density
$$
\widehat{f}_{M}(u,\lambda)= \frac{1}{H_{M}}\left|\sum_{t=1}^{M}h\left(\frac{t}{M}\right)X_{\lfloor uT\rfloor+t-M/2,T}\rme^{-\rmi\lambda\,t}\right|^2
$$
We conclude by applying
Lemmas~\ref{lemma:eigenvalues_covariance_matrix-spectral_density}~\ref{item:1}
and~\ref{lemma:roots_characteristic_polynomial_YW}.
\end{proof}	

\section{Bounds of the estimation risk in the general setting} \label{section:proof_bounds_estimation_risk}
\subsection{Proof of Lemma~\ref{lemma:theta_star_theta_control}}
\label{sec:proof-lemma-refl}
Let us first bound the approximation error $\Gamma_{t/T}-\Gamma_{t,T}^{*}$.
\begin{multline}
\left|\gamma\left(\frac{t}{T},i-j\right)- \gamma^{*}\left(t-i,T,j-i\right)\right| \leq \left|\gamma\left(\frac{t-i}{T},j-i\right)-\gamma\left(\frac{t}{T},i-j\right)\right| \\ 
+ \left|\gamma\left(\frac{t-i}{T},j-i\right)- \gamma^{*}\left(t-i,T,j-i\right)\right|\;. \label{equation:gamma_t_i_gamma_t}
\end{multline}
The second line term of Inequality~(\ref{equation:gamma_t_i_gamma_t}) is upper
bounded by $RT^{-\min(1,\beta)}$ because of Inequality~(\ref{equation:gamma_gamma_star}). Using
the local covariance expression~(\ref{equation:local_covariance_function}),
Cauchy-Schwartz inequality and
$f(,\lambda)\in\Lambda_{1}(\beta,R)$, the following holds for $T\geq d\geq|i|$,
\begin{eqnarray}
\left|\gamma\left(\frac{t-i}{T},j-i\right)-\gamma\left(\frac{t}{T},i-j\right)\right|
&=& \left|\int_{-\pi}^{\pi}\rme^{\rmi\left(j-i\right)\lambda}
\left(f\left(\frac{t-i}{T},\lambda\right)-f\left(\frac{t}{T},\lambda\right)\right)\rmd\lambda\right| \nonumber \\
&\leq& 2\pi dR\,T^{-\min(1,\beta)} \label{equation:gamma_t_i_gamma_t_local_covariance_function}\;.
\end{eqnarray}
Inequality~(\ref{equation:gamma_gamma_star}) implies that
$\smash{\|\bm \gamma_{t,T}^{*}-\bm \gamma_{t/T}\|\leq
  d^{1/2}R\,T^{-\min(1,\beta)}}$ and
inequalities~(\ref{equation:gamma_t_i_gamma_t}),
~(\ref{equation:gamma_t_i_gamma_t_local_covariance_function}) and
again~(\ref{equation:gamma_gamma_star}) imply that for $T\geq d$,
$$
\|\Gamma_{t/T}-\Gamma_{t,T}^{*}\|\leq d(2\pi d R+R)T^{-\min(1,\beta)}\;. 
$$
The smallest
eigenvalue of the matrix $\Gamma_{t/T}$ is greater or equal to $2\pi f_{-}$
(see Lemma~\ref{lemma:eigenvalues_covariance_matrix-spectral_density}~\ref{item:2}). Observe that
\begin{multline*}
\inf_{t}\inf_{\|\bm{a}\|=1}\bm{a}'\Gamma_{t,T}^{*}\bm{a} = \inf_{t}\inf_{\|\bm{a}\|=1}\left\{\bm{a}'\left(\Gamma_{t,T}^{*}-\Gamma_{t/T}\right)\bm{a}+\bm{a}'\Gamma_{t/T}\bm{a}\right\} \\
\geq \inf_{t}\inf_{\|\bm{a}\|=1}\bm{a}'\left(\Gamma_{t,T}^{*}-\Gamma_{t/T}\right)\bm{a} + \inf_{t}\inf_{\|\bm{a}\|=1}\bm{a}'\Gamma_{t/T}\bm{a}
\geq 2\pi f_{-}-d^{3/2}CT^{-\min(1,\beta)}\,.
\end{multline*}
Then, for $T\geq T_{0}=(dR(2\pi d+1)/(2\pi f_{-}))^{1/\min(1,\beta)}$, we have that
$\Gamma_{t,T}^{*}$ is invertible and
$\smash{\|(\Gamma_{t,T}^{*})^{-1}\|\leq(\pi
  f_{-})^{-1}}$. Now, from equations~(\ref{equation:btheta_star}) and~(\ref{equation:definition_theta_Y-W}) we obtain that 
\begin{eqnarray*}
  \btheta_{t,T}^{*}-\btheta_{t/T} &=&
  \left(\Gamma_{t,T}^{*}\right)^{-1}\left[\left(\Gamma_{t/T}-\Gamma_{t,T}^{*}\right)\btheta_{t/T}+
\bm \gamma_{t,T}^{*}-\bm \gamma_{t/T}\right] \;. 
\end{eqnarray*}
Applying matrix inequalities (specifically with the spectral norm) we get 
\begin{equation*}
\|\btheta_{t,T}^{*}-\btheta_{t/T}\| \leq \left\|\left(\Gamma_{t,T}^{*}\right)^{-1}\right\|\left(\left\|\Gamma_{t/T}-\Gamma_{t,T}^{*}\right\|\left\|\btheta_{t/T}\right\|+\left\|\bm \gamma_{t,T}^{*}-\bm \gamma_{t/T}\right\|\right) \;. 
\end{equation*}
Lemma~\ref{lemma:roots_characteristic_polynomial_YW} ensures that $\|\btheta_{t/T}\|\leq 2^{d}$ and the
result follows with $C_{1}=(\pi f_{-})^{-1}(d(2\pi dR+R)2^{d}+R)$.

\subsection{Bias Approximation}
The following elementary lemma will be useful.
\begin{lemma}\label{lem:h-approx-riemann}
  Let $h:[0,1]\to\R$ satisfying~\ref{hypothesis:h}. Then, for all
  $\ell=0,1,2,\dots$ and $M\geq j\geq0$, we have
$$
\frac1M\sum_{s=j+1}^{M}h\left(\frac{s}{M}\right)h\left(\frac{s-j}{M}\right)\left(\frac
  sM\right)^{\ell}=\int_0^1h^2(u)u^\ell\,\rmd u+O_{j,\ell,\|h\|_{\infty},\|h'\|_{\infty}}\left(M^{-1}\right)\;.
$$
In particular, in the case $j=\ell=0$, $H_M=M+O_{\|h\|_{\infty},\|h'\|_{\infty}}\left(1\right)$.
\end{lemma}
\begin{proof}
  The proof is straightforward using, for $M$ large enough, Riemann approximations on the blocks
  defined by $s/M\in(u_{i-1}+j/M,u_i]$, for $i=1,\dots, N$, and neglecting the
  terms from the indices $s$ such that $s/M\in(u_{i-1},u_{i-1}+j/M]$, the
  number of which is bounded. 
\end{proof}
We can now derive the following approximation of the bias.
\begin{theorem} \label{theorem:expectation_J_M} Suppose that
  Assumption~\ref{hypothesis:locally_stationary_with_f} holds with some
  $\beta>0$ and $R>0$, and let $h:[0,1]\to\R$
  satisfying~\ref{hypothesis:h}. Let $k\in\N$ and $\alpha\in(0,1]$ be uniquely
  defined such that $\beta=k+\alpha$. Then, for all $j\in\Z$ and $M\in 2\N^{*}$, we have
\begin{eqnarray*}
\E\left[\widehat{\gamma}_{T,M}\left(\frac{t}{T},j\right)\right] = \gamma\left(\frac{t}{T},j\right) + \sum_{\ell=1}^{k}c_{h,f,j,\ell}\left(\frac{M}{T}\right)^{\ell}+O_{j,\|h\|_{\infty},\|h'\|_{\infty},\beta,R}\left(\frac{1}{M}+\left(\frac{M}{T}\right)^{\beta}\right)\;,
\end{eqnarray*}
where $c_{h,f,j,\ell}\in\C$ only depends on $h$, the spectral density $f$, $j$ and $\ell$. If $h(x)=h(1-x)$ for all $x\in[0,1]$, then $c_{h,f,j,1}=0$.
\end{theorem}
\begin{proof} 
  Without loss of generality we let moreover assume $M\geq j\geq 0$, in which case
\begin{eqnarray*} 
	\E\left[\widehat{\gamma}_{T,M}\left(\frac{t}{T},j\right)\right] &=& 
\frac{1}{H_{M}}\sum_{s=j+1}^{M}h\left(\frac{s}{M}\right)h\left(\frac{s-j}{M}\right)\gamma^{*}\left(t+s-\frac{M}{2},T,j\right)\;. \label{equation:gamma_hat_gamma_star}
\end{eqnarray*}
Since $|h|$ is also piecewise continuously differentiable,
Lemma~\ref{lem:h-approx-riemann} gives that
$$
\frac{1}{|H_{M}|}\sum_{s=j+1}^{M}\left|h\left(\frac{s}{M}\right)h\left(\frac{s-j}{M}\right)\right|=O_{j,\|h\|_{\infty},\|h'\|_{\infty}}(1)\;.
$$
With Inequality~(\ref{equation:gamma_gamma_star}), we obtain that
\begin{align} \label{equation:gamma_gamma_star_R}
  	\E\left[\widehat{\gamma}_{T,M}\left(\frac{t}{T},j\right)\right]&=\gamma_{M,j}
+O_{j,\|h\|_{\infty},\|h'\|_{\infty},R}\left(T^{-\min(1,\beta)}\right)\;,
\\
\nonumber\text{where}\quad &
	\gamma_{M,j} := \frac{1}{H_{M}}\sum_{s=j+1}^{M}h\left(\frac{s}{M}\right)h\left(\frac{s-j}{M}\right)\gamma\left(\frac{t+s-M/2}{T},j\right) \;.
      \end{align}
Since $f(\cdot,\lambda)\in\Lambda_{1}(\beta,R)$, a Taylor expansion yields
\begin{eqnarray*}
	f\left(\frac{t-M/2+s}{T},\lambda\right) =
        \sum_{\ell=0}^{k}\frac{\partial_{1}^{\ell}
          f\left(t/T,\lambda\right)}{\ell!}\left(\frac{-M/2+s}{T}\right)^{\ell}
        +f_{k}\left(\frac tT,\frac{-M/2+s}{T},\lambda\right)\;,
\end{eqnarray*}
with
\begin{equation}
  \label{eq:remainder-taylor-local-spdens}
\sup_{t\in\Z,1\leq s\leq M}\left|f_{k}(t/T,(-M/2+s)/T,\lambda)\right|=O_{\beta,R}((M/T)^{\beta})\;.  
\end{equation}
Then
\begin{multline}
\gamma_{M,j} = \frac{1}{H_{M}}\int_{-\pi}^{\pi}\rme^{\rmi j\lambda}\sum_{s=j+1}^{M}h\left(\frac{s}{M}\right)h\left(\frac{s-j}{M}\right)f\left(\frac{t-M/2+s}{T},\lambda\right)\rmd\lambda = \\
\sum_{\ell=0}^{k}\int_{-\pi}^{\pi}\frac{\partial_{1}^{\ell} f\left(t/T,\lambda\right)}{\ell!}\rme^{\rmi j\lambda}\frac{1}{H_{M}}\sum_{s=j+1}^{M}h\left(\frac{s}{M}\right)h\left(\frac{s-j}{M}\right)\left(\frac{-M/2+s}{T}\right)^{\ell}\rmd\lambda \\
+\int_{-\pi}^{\pi}\rme^{\rmi j\lambda}\frac{1}{H_{M}}\sum_{s=j+1}^{M}h\left(\frac{s}{M}\right)h\left(\frac{s-j}{M}\right)f_{k}\left(\frac
  tT,\frac{-M/2+s}{T},\lambda\right)\rmd\lambda\;. \label{equation:integral_phi_h_f_after_Taylor}
\end{multline}
Lemma~\ref{lem:h-approx-riemann} yields that for all $\ell=1,\ldots,k$,
\begin{multline*}
\frac{1}{H_{M}}\sum_{s=j+1}^{M}h\left(\frac{s}{M}\right)h\left(\frac{s-j}{M}\right)\left(\frac{-M/2+s}{T}\right)^{\ell}
= \\ \left(\frac{M}{T}\right)^{\ell}
\frac{M}{H_{M}}\frac{1}{M}\sum_{s=j+1}^{M}h\left(\frac{s}{M}\right)h\left(\frac{s-j}{M}\right)\left(-\frac{1}{2}+\frac{s}{M}\right)^{\ell}
 = c_{h,\ell}\left(\frac{M}{T}\right)^{\ell}+O_{j,\|h\|_{\infty},\|h'\|_{\infty},\ell}\left(\frac{1}{M}\right)\left(\frac{M}{T}\right)^{\ell}\;, 
\end{multline*}
with $\displaystyle c_{h,\ell}=\int_{0}^{1}h^{2}\left(u\right)\left(u-\frac{1}{2}\right)^{\ell}\rmd u$.
Observe that $c_{h,0}=1$ by assumption in~\ref{hypothesis:h}, and, if $h(x)=h(1-x)$ for all
$x\in[0,1]$, then moreover $c_{h,1}=0$. From this it follows that
\begin{align*}
\int_{-\pi}^{\pi}\frac{\partial_{1}^{\ell} f\left(t/T,\lambda\right)}{\ell!}\frac{\rme^{\rmi j\lambda}}{H_{M}}\sum_{s=j+1}^{M}h\left(\frac{s}{M}\right)h\left(\frac{s-j}{M}\right)\left(\frac{-M/2+s}{T}\right)^{\ell}\rmd\lambda = \left(\frac{M}{T}\right)^{\ell} \left(c_{h,f,j,\ell}
+ O_{j,\|h\|_{\infty},\|h'\|_{\infty},\beta,R}\left(\frac{1}{M}\right)\right)\;, 
\end{align*}
where, in particular, $c_{h,f,j,0}=\gamma(t/T,j)$ and $c_{h,f,j,1}=0$ if
$h(x)=h(1-x)$. Finally, by~(\ref{eq:remainder-taylor-local-spdens}), and since
$|h|$ is also piecewise continuously differentiable, the reminder term in~(\ref{equation:integral_phi_h_f_after_Taylor}) satisfies
$$
\int_{-\pi}^{\pi}\rme^{\rmi
  j\lambda}\frac{1}{H_{M}}\sum_{s=j+1}^{M}h\left(\frac{s}{M}\right)h\left(\frac{s-j}{M}\right)f_{k}\left(\frac{t-M/2+s}{T},\lambda\right)\rmd\lambda
=
O_{j,\|h\|_{\infty},\|h'\|_{\infty},\beta,R}\left(\left(\frac{M}{T}\right)^{\beta}\right)\;.
$$
We thus finally obtain that
\begin{eqnarray*}
	\gamma_{M,j} = \gamma\left(\frac{t}{T},j\right) +\sum_{\ell=1}^{k}c_{h,f,j,\ell}\left(\frac{M}{T}\right)^{\ell}+O_{d,\|h\|_{\infty},\|h'\|_{\infty},\beta,R}\left(\frac{1}{M}+\left(\frac{M}{T}\right)^{\beta}\right)\;.
\end{eqnarray*}
This approximation and the bound~(\ref{equation:gamma_gamma_star_R}) allow us to conclude the proof.
\end{proof}
\subsection{Proof of Theorem~\ref{theorem:thetas_difference_development}} \label{section:proof_theorem_thetas_difference_development}
Gathering together Assumption~\ref{hypothesis:gamma} and
Theorem~\ref{theorem:expectation_J_M} yields, for any $j=-d,\dots,d$,
\begin{multline} \label{equation:J_M_J}
\widehat{\gamma}_{T,M}\left(\frac{t}{T},j\right) = \gamma\left(\frac{t}{T},j\right) + \sum_{\ell=1}^{k}c_{h,f,j,\ell}\left(\frac{M}{T}\right)^{\ell}+O_{d,\|h\|_{\infty},\|h'\|_{\infty},\beta,R}\left(\frac{1}{M}+\left(\frac{M}{T}\right)^{\beta}\right) \\+ u_{M}\left(\frac{t}{T},j\right)\;,
\end{multline}
where $u_{M}(t/T,j) = O_{L^{\bullet}(\mu),d,\|h\|_{\infty},\|h'\|_{\infty},C}(M^{-1/2})$ and $c_{h,f,j,1}=0$ if $h(x)=h(1-x)$ for all $x\in[0,1]$.

For the sake of simplicity, we drop $t,T$ in the notation and set
$\bm\gamma=\bm\gamma_{t/T}$,
$\widehat{\bm\gamma}_{M}=\widehat{\bm\gamma}_{t,T,M}$,
$\Gamma:=\Gamma_{t/T}$ and
$\widehat{\Gamma}_{M}:=\widehat{\Gamma}_{t,T,M}$. Using the
expression~(\ref{equation:J_M_J}), we obtain
\begin{align}
\Gamma-\widehat{\Gamma}_{M} &= \sum_{\ell=1}^{k}C_{h,f,\ell}\left(\frac{M}{T}\right)^{\ell}+O_{d,\|h\|_{\infty},\|h'\|_{\infty},\beta,R}\left(\frac{1}{M}+\left(\frac{M}{T}\right)^{\beta}\right) + U_{M}\;, \label{equation:Gammas_difference_development} \\
\widehat{\bm\gamma}_{M}-\bm\gamma &= \sum_{\ell=1}^{k}\bm c_{h,f,\ell}\left(\frac{M}{T}\right)^{\ell}+O_{d,\|h\|_{\infty},\|h'\|_{\infty},\beta,R}\left(\frac{1}{M}+\left(\frac{M}{T}\right)^{\beta}\right) + \bm u_{M}\;, \label{equation:gammas_difference_development}
\end{align} 
where the matrices $C_{h,f,\ell}\in\R^{d\times d}$ and the vectors $\bm
c_{h,f,\ell}\in\R^{d}$ only depend on $d,h,f$ and $\ell$. Furthermore $U_{M}$ and
$\bm u_{M}$ both are $O_{L^{\bullet}(\mu),d,\|h\|_{\infty},\|h'\|_{\infty},C}(M^{-1/2})$. Again
$C_{h,f,1}=0$ and $\bm c_{h,f,1}=0$ if $h(x)=h(1-x)$ for all $x\in[0,1]$.

Note now that Lemma~\ref{lemma:eigenvalues_covariance_matrix-spectral_density}
with the assumption $f_->0$ further says that $\Gamma_{t/T}$ is also
non-singular with $\Gamma_{t/T}^{-1}=O_{f_-}(1)$ and,
with~(\ref{equation:definition_hat_theta_Y-W_allcases}), we can thus apply
Lemma~\ref{lemma:thetas_difference_development}, showing
that~(\ref{equation:thetas_difference_development_k}) holds with
$\btheta=\btheta_{t/T}$,
$\widehat{\btheta}=\widehat{\btheta}_{t,T}\left(M\right)$,
$\bm\gamma=\bm\gamma_{t/T}$,
$\widehat{\bm\gamma}=\widehat{\bm\gamma}_{M}=\widehat{\bm\gamma}_{t,T,M}$,
$\Gamma:=\Gamma_{t/T}$ and
$\widehat{\Gamma}=\widehat{\Gamma}_{M}:=\widehat{\Gamma}_{t,T,M}$.  Next the
bounds provided by Lemmas~\ref{lemma:roots_characteristic_polynomial_YW}
and~(\ref{equation:definition_hat_theta_Y-W_allcases}) further imply $\|\btheta\|$ and
$\|\widehat{\btheta}\|$ to be less that $2^d$.  This, with the
approximations~(\ref{equation:Gammas_difference_development})
and~(\ref{equation:gammas_difference_development}),
yields~(\ref{equation:thetas_difference_development}).
\subsection{Proof of
  Theorem~\ref{theorem:combined_thetas_difference_development}} \label{section:proof_theorem_combined_thetas_difference_development}
For each $j=0,\ldots,k$ replace $M$ by $M/2^{j}$
in~(\ref{equation:thetas_difference_development}), multiply the resulting
expression by $\omega_{j}$ and sum over $j$. Matrix $A$ (definition below
Equation~(\ref{equation:alpha})) is a non singular Vandermonde matrix and
$\bm\omega$ is well defined. 
\section{Postponed proofs for  TVCBS and TVAR processes}
\subsection{Proof of Theorem~\ref{thm-ass-C-for-CBS}}\label{sec:proof-CBS-locally-stationary-proof}
  Let us denote, for any $L^p$ random variable $Z$,
  $\|Z\|_p=\left(\E\left[\left|Z\right|^{p}\right]\right)^{1/p}$ its $L^p$-norm.
  The proof relies on the Burkhölder inequality for non-stationary dependent
  sequences. Namely, an immediate consequence of
  \cite[Proposition~5.4]{Dedecker_Doukhan_Lang:2007} and the Hölder and
  Minkowskii inequalities is that if $p\geq2$ and
  $(Z_t)_{t\in\zset}$ is a 
  $L^p$ process adapted to the filtration $(\mathcal{F}_t)_{t\in\zset}$, then,
  for all $s\in\zset$ and $n\geq1$
$$
\left\|\sum_{t=s+1}^{s+n}(Z_t-\E[Z_t])\right\|_p^{2}
\leq 2 p
n\,\left(\sup_{t\in\Z}\left\|Z_t-\E[Z_t]\right\|_p\right)
\,\left(\sup_{t\in\Z}\sum_{k=0}^\infty\left\|\E\left[Z_{t+k}|\mathcal{F}_t\right]-\E[Z_{t+k}]\right\|_p\right)\;.
$$
Applying this inequality with $s=|\ell|+1$, $n=M-|\ell|$ to 
$$
Z_t=h\left(\frac{t}{M}\right)h\left(\frac{t-|\ell|}{M}\right)X_{\lfloor
  uT\rfloor+t-M/2,T}X_{\lfloor uT\rfloor+t-|\ell|-M/2,T},\quad\mathcal{F}_t=\mathcal{F}^\xi_{\lfloor
  uT\rfloor+t-M/2}\;,
$$
where $\mathcal{F}^\xi_t=\sigma(\xi_s,s\leq t)$ denotes the natural filtration
of $(\xi_t)_{t\in\Z}$, we obtain that, for any $M>|\ell|$,
\begin{multline*}
\left\|\widehat{\gamma}_{T,M}\left(u,\ell\right)-\E\left[\widehat{\gamma}_{T,M}\left(u,\ell\right)\right]\right\|_p
\leq \frac{2(p(M-|\ell|))^{1/2}}{H_M}\,\|h\|_\infty^2\;\sup_{s\in\zset}\|X_{s,T}\|_{2p}\\
\times\left(\sup_{t\in\zset}\sum_{k=0}^\infty\left\|\E\left[X_{t+k,T}X_{t-|\ell|+k,T}\mid
    \mathcal{F}^\xi_{t}\right]-\E[X_{t+k,T}X_{t-|\ell|+k,T}]\right\|_p \right)^{1/2}\;. 
\end{multline*}
Under Assumption~\ref{hypothesis:h}, $H_{M}/M\in(1/2,3/2)$ for $M$ large
enough (see Lemma~\ref{lem:h-approx-riemann})
and it is thus now sufficient to show that, for any $p\geq2$, 
\begin{align}
  \label{eq:to-show-burk1}
&\sup_{s\in\zset}\|X_{s,T}\|_{2p}=O_{K,r,(\psi_k)}(1)\\
  \label{eq:to-show-burk2}
&\sup_{t\in\zset}\sum_{k=0}^\infty\left\|\E\left[X_{t+k,T}X_{t-|\ell|+k,T}\mid \mathcal{F}^\xi_{t}\right]-\E[X_{t+k,T}X_{t-|\ell|+k,T}]\right\|_p=O_{K,r,k_0,\ell,(\psi_k),(\zeta_k)}(1)
\end{align}
The bound~(\ref{eq:to-show-burk1}) is a direct consequence
of~(\ref{eq:CBS-def}) and~(\ref{eq:CBS-def-1}) and the assumption on
$(\xi_t)_{t\in\Z}$.  To prove~(\ref{eq:to-show-burk2}), let us define, for all
$t\in\zset$, $T\geq T_0$ and $\bm x\in\R^\N$,
$$
\Phi^0_{t,T}(\bm x)=\varphi^0_{t,T}(\bm x)\,\varphi^0_{t-|\ell|,T}((x_{j+|\ell|})_{j\geq0})\;.
$$
It then follows that $X_{t,T}X_{t-|\ell|,T}=\Phi^0_{t,T}((\xi_{t-j})_{j\geq0})$
and is then straightforward to show that the assumptions on $\varphi^0$ yields that
there exists some constant $K'$ only depending on $K$ and $(\psi_j)_{j\in\N}$ and
$(\zeta_j)_{j\in\N}$ such that,  for all $t\in\zset$, $T\geq T_0$
  and all $\bm x,\bm x'\in\R^\N$ satisfying $x_j=x'_j$ for $1\leq j\leq k_0+|\ell|$,
\begin{align*}
&\left|\Phi^0_{t,T}\left(\bm x\right)\right| \leq K'\left(1+\sum_{j=0}^{\infty}\tilde\psi_j\left|x_{j}\right|\right)^{2r}\;, \\  
&
\left|\Phi^0_{t,T}(\bm x)-\Phi^0_{t,T}(\bm x')\right|\leq K'\,
\left(\sum_{j\geq0}\tilde\zeta_j|x_{k_0+j}-x'_{k_0+j}|\right)\,\left(1+\sum_{j\geq0}\tilde\psi_j(|x_j|+|x'_j|)\right)^{2r-1}\;,
\end{align*}
where $\tilde\psi_j=\psi_j$ and $\tilde\zeta_j=\zeta_j$ for $0\leq j<|\ell|$,
and $\tilde\psi_j=\max(\psi_{j-|\ell|},\psi_j)$ and
$\tilde\zeta_j=\zeta_{j-|\ell|}+\zeta_j$ for $j\geq |\ell|$.  By Jensen's
inequality, we have that, for all $k\geq0$,
$$
\left\|\E\left[X_{t+k,T}X_{t-|\ell|+k,T}\mid
  \mathcal{F}^\xi_{t}\right]-\E[X_{t+k,T}X_{t-|\ell|+k,T}]\right\|_p
\leq\left\|\Phi^0_{t+k,T}((\xi_{t+k-j})_{j\geq0})-\Phi^0_{t+k,T}((\xi'_{t+k-j})_{j\geq0})\right\|_p\;,
$$
where $(\xi'_s)_{s\in\zset}$ is i.i.d. with same distribution as
$(\xi_s)_{s\in\zset}$ and such that $\xi_{t+k-j}=\xi'_{t+k-j}$ for all
$j=0,1,\dots,k-1$ and $(\xi'_{t-j})_{j\geq0}$ and  $(\xi_{t-j})_{j\geq0}$ are
independent. With the above bounds on $\Phi^0$ and using the Minkowskii and
Hölder inequalities, we thus get that, for all
$t\in\zset$ and $k\geq0$,
$$
\left\|\E\left[X_{t+k,T}X_{t-|\ell|+k,T}\mid
  \mathcal{F}^\xi_{t}\right]-\E[X_{t+k,T}X_{t-|\ell|+k,T}]\right\|_p\leq 2K'
(1+\|\tilde{\bm \psi}\|_1\|\xi_0\|_{2rp})^{2r}
\;,
$$
and, if $k> k_0+|\ell|$, the same quantity is bounded from above by
$$
2K'\|\xi_0\|_{2p}\,\left(1+2\|\tilde{\bm \psi}\|_1\|\xi_0\|_{2p(2r-1)}\right)^{2r-1}
\left(\sum_{j=k-k_0}^\infty\tilde\zeta_j\right)\;.
$$
Summing these bounds over all $k\geq0$, we obtain
\begin{multline*}
\sum_{k=0}^\infty\left\|\E\left[X_{t+k,T}X_{t-|\ell|+k,T}\mid
    \mathcal{F}^\xi_{t}\right]-\E[X_{t+k,T}X_{t-|\ell|+k,T}]\right\|_p\\
\leq 2 K'\,\left(1+2\|\tilde{\bm \psi}\|_1\|\xi_0\|_{4rp}\right)^{2r}\,\left(k_0+|\ell|+
\sum_{k> k_0+|\ell|}\sum_{j=k-k_0}^\infty\tilde\zeta_j\right)\\
= 2 K'\,\left(1+2\|\tilde{\bm \psi}\|_1\|\xi_0\|_{4rp}\right)^{2r}\,\left(k_0+|\ell|+
\sum_{j> |\ell|}(j-|\ell|)\tilde\zeta_j\right)\;.  
\end{multline*}
By~(\ref{eq:lip-cond-CBS}), the term between parentheses is a finite constant only depending on $k_0$,
$|\ell|$ and $\bm \zeta$. We thus have shown~(\ref{eq:to-show-burk2}), which
concludes the proof.

\subsection{Proof of Theorem~\ref{theorem:A_TVAR}} \label{section:results_tvar} We 
first need to recall some basic facts on the representation of a TVAR process
as a TVCBS. Let us set $\bm e_{1}=[1\; 0\dots 0]'\in\R^p$ and introduce the
companion $p\times p$ matrices defined for all $u\in\R$ by
$$
A(u)=
\begin{bmatrix}
  \theta_1(u) &\theta_2(u)&\dots&\dots &\theta_p(u) \\
  1 &0&\dots&\dots&0 \\
  0 &1&0 &\dots&0 \\
  \vdots&\ddots&\ddots&\ddots&\vdots\\
  0&\dots&0&1&0
\end{bmatrix}\;.
$$
By \cite[Proposition~1 and its proof]{Giraud_Roueff_Sanchez-Perez:2015}, the
TVAR process $(X_{t,T})_{t\in\Z,T\geq T_0}$ is a special case of TVCBS
introduced in Example~\ref{example:non-stationary_CBS}. Namely, it satisfies a
representation~(\ref{eq:CBS-def}) with a linear form
$$
\varphi^0_{t,T}(\bm x)=\sum_{j=0}^\infty a_{t,T}(j)\sigma((t-j)/T)\; x_j \;,
\quad\text{with}\quad
a_{t,T}(j)=\bm e'_{1}\left[\prod_{i=1}^{j}A\left(\Dfrac{t-i}{T}\right)\right]\bm e_{1}\;,
$$
and moreover, there exist some constants $\bar K>0$ and $\delta_1\in(0,1)$ only depending on
$p,\delta,\beta$ and $R_0$ such that, for all $T\geq T_0$, $t\in\Z$ and $j\in\N^*$, 
\begin{equation}
  \label{eq:tvar-prod-matrices-bound}
\left\|\prod_{i=1}^{j}A\left(\Dfrac{t-i}{T}\right)\right\|\leq \bar K\,\delta_{1}^j \;.  
\end{equation}
Hence in particular 
$$
|a_{t,T}(j)|\leq \bar  K \delta_1^j\;,
$$
implying that $\varphi^0$ above satisfies~(\ref{eq:CBS-def-1})
and~(\ref{eq:lip-cond-CBS-rep}) with $r=1$, $K=1$ and
$\psi_j=\zeta_j=\bar K\sigma_+ \delta_1^j$. 

We can now proceed with the proof of Theorem~\ref{theorem:A_TVAR}, starting with~\ref{item:tvar-1}.
To show that it is weakly locally stationary
with the local spectral density of the AR($p$) with local standard deviation
$\sigma(u)$ and autoregresive coefficients $\theta_1(u),\dots,\theta_p(u)$, it
only remains to show that~(\ref{eq:CBS-def-2}) 
holds with $r=1$ and
$$
\varphi(u,\bm x)=\sigma(u)\sum_{j=0}^\infty \bm e'_{1}A^j\left(u\right)\bm e_{1}\; x_j \;.
$$
This is a simple consequence of~(\ref{eq:tvar-prod-matrices-bound}) and the
fact that $\sigma\in\Lambda_1(\min(1,\beta),R_0)$ and
$\bm\theta\in\Lambda_p(\min(1,\beta),R_0)$ by assumption. Then~(\ref{eq:pred-tvar-allcoeff-same}) follows from the
representation~(\ref{eq:CBS-def})~: using that it is causal
and~(\ref{equation:TVAR-p}), we get in the case $d=p$ that
$\btheta(t/T)=\btheta_{t,T}^*$. On the other hand, the mere definition of the
local spectral density in~(\ref{eq:tvar-local-dens}) (being that of an AR($p$)
process at fixed $u$), the case $d=p$ yields that $\btheta_{u}=\btheta(u)$.
Finally, an additional consequence of the definition of $f$ is that
the spectral density $f(\cdot,\lambda)$ belongs to $\Lambda_{1}(\beta,R)$ for
any $\lambda\in\R$ with $R$ only depending on $p,\delta,\beta,\sigma_+$ and
$R_0$, provided that we can show that
$\left|\btheta\left(\rme^{-\rmi\lambda};u\right)\right|$ can be bounded from
below by a positive constant only depending on $\delta$ and $p$. 
By Lemma~\ref{lem:basic-s_pdelta} and continuity of
$(\lambda,\btheta)\mapsto1-\sum_{j=1}^{p}\theta_{j}\rme^{-\rmi j\lambda}$, and
since $\delta^{-1}>1$ we have that
\begin{equation}
  \label{eq:bounds-polynomial-spdelta}
0<\inf_{\btheta\in
  s_{(p)}\left(\delta\right),\lambda\in\R}\left|1-\sum_{j=1}^{p}\theta_{j}\rme^{-\rmi
    j\lambda}\right|
\leq \sup_{\btheta\in
  s_{(p)}\left(\delta\right),\lambda\in\R}\left|1-\sum_{j=1}^{p}\theta_{j}\rme^{-\rmi
    j\lambda}\right| <\infty\;.  
\end{equation}
Of course these two constants only depend on $\delta$ and $p$ and the inf one
can serve as a lower bound of
$\left|\btheta\left(\rme^{-\rmi\lambda};u\right)\right|$, concluding the proof
of~\ref{item:tvar-1}.  

Next, we prove~\ref{item:tvar-1bis} and~\ref{item:tvar-1ter}, which respectively require the two add-on
properties
\begin{enumerate}[label=(\alph*)]
\item\label{item:fact1}  $f(u,\lambda)\geq f_{-}$ for all $u,\lambda\in\R$, 
\item\label{item:fact2} if  $\xi_0$ has a diffuse
distribution, then $\PP(X_{t,T}=0)=0$ for all $t\in\zset$ and $T\geq T_0$. 
\end{enumerate}
The fact~\ref{item:fact1} follows from~(\ref{eq:tvar-local-dens}),
$\sigma(u)\geq\rho\sigma_+$ and the upper bounds
in~(\ref{eq:bounds-polynomial-spdelta}), which shows that we can find such an
$f_->0$ only depending on $p$, $\delta$, $\rho$ and
$\sigma_+$. Fact~\ref{item:fact2} is a consequence of the TVCBS representation
above and the assumptions on $(\xi_t)$ which implies that for all $t,T$ and
$j\in\N$, we can write $X_{t,T}=a_{t,T}(j)\xi_{t-j}+Z_{t,T}(j)$ with
$Z_{t,T}(j)$ independent of $\xi_{t-j}$. Hence, if $\xi_0$ has a diffuse
distribution, it only remains to prove that for all $T\geq T_0$ and all
$t\in\Z$, there exists $j\in\N$ such that $a_{t,T}(j)\neq0$. Using again the 
 TVCBS representation
above, this is equivalent to show that  for all $T\geq T_0$ and all
$t\in\Z$, $\gamma^{*}\left(t,T,0\right)=\mathrm{var}(X_{t,T})>0$. Now observe
that by~\ref{item:fact1} and since  $(X_{t,T})_{t\in\Z,T\geq T_0}$ is
$(\beta,R)$-weakly locally stationary, we have 
$$
\gamma^{*}\left(t,T,0\right)\geq \gamma(t/T,0)-R\,T^{-\min(1,\beta)}\geq 2\pi f_--R\,T^{-\min(1,\beta)}>0\;,
$$
where the last inequality holds by taking $T_0$ large enough (only depending on
$f_-$ and $R$ and thus on    $p,\delta,\rho,\sigma_+$ and $R_0$.). 

We conclude with the proof of~\ref{item:tvar-2}. We use the above TVCBS
representation which were mentioned to satisfy~(\ref{eq:lip-cond-CBS-rep}) with
$\zeta_j=\bar K\sigma_+ \delta_1^j$ for some
$\delta_1<1$. Hence~(\ref{eq:lip-cond-CBS}) holds as well and we can apply
Theorem~\ref{thm-ass-C-for-CBS}, which shows that the TVAR process
satisfies~\ref{hypothesis:gamma}.

\end{document}